\documentclass[review]{elsarticle}

\usepackage{amssymb,amsmath,comment}
\usepackage{amsthm}
\usepackage[usenames,dvipsnames]{xcolor}
\usepackage{subfigure}
\usepackage{caption}
\usepackage{algorithm, algorithmicx, algcompatible}
\usepackage{algpseudocode}
\usepackage{graphicx}
\usepackage{tikz}
\usepackage[export]{adjustbox}
\usepackage{epstopdf}
\usepackage{hyperref}
\usepackage{multirow}
\usepackage{diagbox}
\usepackage{enumitem}

\definecolor{cadmiumgreen}{rgb}{0.0, 0.42, 0.24}

\usepackage{lineno}
\modulolinenumbers[5]
\newdefinition{rmk}{Remark }
\newdefinition{prop}{Proposition}
\newtheorem{lemma}{Lemma}
\newtheorem{theorem}{Theorem}

\makeatletter
\def\@author#1{\g@addto@macro\elsauthors{\normalsize%
    \def\baselinestretch{1}%
    \upshape\authorsep#1\unskip\textsuperscript{%
      \ifx\@fnmark\@empty\else\unskip\sep\@fnmark\let\sep=,\fi
      \ifx\@corref\@empty\else\unskip\sep\@corref\let\sep=,\fi
      }%
    \def\authorsep{\unskip,\space}%
    \global\let\@fnmark\@empty
    \global\let\@corref\@empty  
    \global\let\sep\@empty}%
    \@eadauthor={#1}
}
\makeatother

\begin{document}
\begin{frontmatter}

\title{Accelerating GMRES with Matrix-Free Multiscale Robin Preconditioners}

\author[]{Dilong Zhou$^1$}
\cortext[cor1]{Corresponding author}
\ead{zdtc68@gmail.com}
\author[]{Rafael T. Guiraldello$^2$}
\ead{rafaeltrevisanuto@gmail.com}
\author[]{Felipe Pereira$^{3,*}$ \corref{cor3}}
\ead{luisfelipe.pereira@utdallas.edu}
\author[]{Fabr\'{\i}cio S. Sousa$^1$}
\ead{f.s.sousa@usp.br}

\address{$^1$ Institute of Mathematics and Computer Science,\\
University of S\~ao Paulo, \\
Av. Trab. S\~ao Carlense, 400 - Centro, S\~ao Carlos, 13566-590, SP, Brazil \\
$^2$ Piri Technologies, LLC \\
 1000 E. University Ave., Dept. 4311, Laramie, WY 82071-2000, USA\\
$^3$ Department of Mathematical Sciences, The University of Texas at Dallas, \\
800 W. Campbell Road, Richardson, Texas 75080-3021, USA \\
}


\begin{abstract}

We propose a matrix-free right-preconditioning strategy for the Generalized Minimal Residual (GMRES) method based on the Multiscale Robin Coupled Method with oversampling (MRCM-OS) for the numerical solution of elliptic problems arising in subsurface flow.
The resulting preconditioner is constructed through local subdomain solves with oversampling and smoothing, and can be applied without explicit assembly of the global operator.

After a careful presentation of the new procedure, it is used in extensive numerical experiments. 
Our results demonstrate that the proposed approach substantially reduces iteration counts across a range of challenging, high-contrast subsurface flow problems. In many cases, convergence is obtained in one or two GMRES iterations when oversampling and smoothing are employed.

The results indicate that combining GMRES with multiscale Robin-based operators is a promising direction for the construction of rapidly convergent preconditioning strategies.



\end{abstract}

\end{frontmatter}

\section{Introduction}\label{intro-pre}
Simulating two- and three-phase flow problems relevant to the oil industry often involves handling high-resolution permeability data, extremely large computational grids with billions of cells, and extended time steps. Within each time step, both a convection-diffusion equation and a second-order elliptic equation must be solved. The presence of high-contrast permeability fields adds further complexity to these problems. Accurately capturing fluid flow within complex rock formations is vital for geoscience and reservoir modeling, yet remains challenging despite considerable research efforts.

Over the last two decades, researchers have developed a range of approaches that combine domain decomposition strategies with multiscale techniques, making significant advancements in the field. These methods can be grouped into three primary types: Multiscale Finite Volume, Multiscale Finite Element, and Multiscale Mixed Finite Element approaches. A comprehensive review of different strategies for developing multiscale methods can be found in \cite{MRCM-OS, zhou2025}. This paper focuses on the Multiscale Robin Coupled Method (MRCM) \cite{guiraldello2018multiscale, pereira} and its related procedures \cite{hani_pereira_1, 2020MUMM}.

The Multiscale Robin Coupled Method (MRCM) applies Robin-type boundary conditions to all local problems defined on a non-overlapping domain partition. Based on a domain decomposition approach \cite{Douglas1993}, it enforces weak continuity of pressure and the normal flux component across adjacent subdomain interfaces. $\alpha$ regulates the balance between pressure continuity and normal flux continuity in the Robin boundary conditions, enabling MRCM to simulate both MMMFEM and MHM \cite{guiraldello2018multiscale}. By prioritizing pressure continuity (adjust $\alpha$ to 0), MRCM approximates MMMFEM. Conversely, by emphasizing flux continuity (adjust $\alpha$ to $\infty$), MRCM converges toward MHM. 

MRCM’s impressive scalability on parallel multicore supercomputers has driven considerable research aimed at refining the method and extending its application to new problems. Relevant examples can be found in \cite{interface2021, MMtwo2020, MRCM2022, MPM1, MPM2}.

Along the interfaces of non-overlapping subdomains, the resonance errors in MRCM-OS are reduced more effectively than in MRCM through the incorporation of oversampling and smoothing procedures. Oversampling enables each Multiscale Basis Function (MBF) to be constructed in extended regions while preserving the compatibility conditions across the non-overlapping partitions. Additionally, smoothing procedures implemented as a Schwarz-type method derived from overlapping Schwarz domain decomposition techniques \cite{DDintro, smith2004domain} address small-scale errors arising in the multiscale solution and significantly improve accuracy. Together, these strategies allow MRCM-OS to achieve a two-order-of-magnitude improvement in accuracy compared to the original MRCM method.

The integration of multiscale coarse spaces has substantially enhanced the robustness and efficiency of traditional preconditioners over the last three decades. This progress has spurred the development of numerous innovative approaches. Among these multiscale preconditioners, the majority \cite{2007robust, 1996nonstandard, 2014algebraic, 2016anadaptive, 2017BDDC, 2012analysis, 2018BDDC, 2015FETI-DP, 2007adaptive, 2010domain1, 2010domain2, 2011acourse, 2016acomparison, 2019adaptive, 2022multilevel, 2015upscaling, 2025ahighly, 2025conferenceschwarz, Jiang20251962, LI2024117056, BOUTILIER2024561, CALVO2024112909,
dryja2010feti, dryja2015analysis,  dryja2016deluxe, yu2023non} have been tailored to elliptic problems formulated in a second-order setting. This paper focuses on such problems.

Various types of preconditioners have been developed, including additive Schwarz preconditioners, which also encompass multilevel additive methods \cite{2007robust, 1996nonstandard, 2012analysis, 2010domain1, 2010domain2, 2011acourse, 2019adaptive, 2022multilevel, 2025conferenceschwarz, CALVO2024112909, 2020amultilevel, 2010spectral, Lu20252187, Heinlein2025A1170, GUILLET2025114136, Martin20242986}. There are also BDDC (Balancing Domain Decomposition by Constraints) and FETI-DP (Finite Element Tearing and Interconnecting - Dual-Primal) preconditioners \cite{2016anadaptive, 2017BDDC, 2018BDDC, 2015FETI-DP, 2007adaptive, 2016acomparison, dryja2010feti, galvis2009feti, galvis2010bdd}, as well as two-grid and two-level multiscale preconditioners \cite{2025ahighly, LI2024117056, BOUTILIER2024561, 2024efficient, 2024arobust, 2025two, 2025high}. Additionally, there are preconditioners with Conjugate Gradient methods (PCG) \cite{Jiang20251962, 2013nonsymmetric} and those employing the Generalized Minimum Residual Method (GMRES) \cite{2014algebraic, 2024optimizing, Martin20242041}. This paper specifically focuses on preconditioners built around GMRES and its modification.

Several studies utilize multiscale techniques to improve accuracy, drawing on methods such as the Multiscale Finite Element Method \cite{2007robust, 2014algebraic, 2010domain1, 2010domain2, 2015upscaling, 2024optimizing}, the Multiscale Finite Volume Method \cite{2014algebraic}, the Multiscale Mixed Finite Element Method \cite{2019domain}, and the Generalized Multiscale Finite Element Method \cite{2024arobust, 2019atwo-grid, 2024anadaptive}. Given the high accuracy demonstrated by MRCM-OS, this paper combines GMRES with MRCM-OS to construct a highly efficient preconditioner.

The main contribution of this work is not a new multiscale discretization, but the use of an oversampled MRCM operator as a matrix-free right preconditioner within GMRES. While multiscale methods based on the Robin boundary condition for subdomain coupling have been previously developed, their use as linear operators defining effective preconditioners for Krylov methods has not been systematically analyzed.
The key idea is to reinterpret the multiscale procedure as an approximate inverse operator acting on residuals, enabling the construction of a fast solver without forming the global matrix.

This work is divided into several sections. The first section introduces the formulation of MRCM, incorporating oversampling and smoothing. The following section describes the preconditioner construction using GMRES and its integration with MRCM-OS to achieve a highly efficient preconditioner. Subsequently, our numerical experiments are presented, with separate subsections focusing on different experimental aspects.
First, we consider two examples with analytical solutions: one with a constant permeability field \cite{MRCM-OS, zhou2025} and one with a variable permeability field. The second set of experiments includes three permeability fields from the SPE10 dataset \cite{SPE10}: one containing a challenging high-permeability channel, one with the highest contrast ratio, and one representing a typical field with moderate contrast and no channel. Finally, we examine the most challenging cases by considering large-contrast permeability fields containing high- and/or low permeability inclusions from \cite{interface2021}, demonstrating how our method performs under highly challenging permeability configurations.
These numerical results indicate very rapid convergence, across the tested cases, of the preconditioning strategy based on the MRCM-OS method. In Appendix A, we formulate a condition that would provide a theoretical explanation for this observed behavior. Verifying whether this condition is satisfied by the proposed method lies beyond the scope of the present work and is deferred to future study.


\section{The MRCM with Oversampling and Smoothing (MRCM-OS)}

In our analysis, we focus on single-phase flow within porous media. The major equations, which describe the pressure $p$ and the Darcy velocity ${\bf u}$, are expressed as follows:

\begin{eqnarray}
{\bf u} & = & -\, K\,\nabla p \qquad \mbox{in}~\Omega \label{eq1a-pre}\\
\nabla \cdot {\bf u} & = & f \qquad \mbox{in}~\Omega \label{eq1b-pre}\\
p & = & g \qquad \mbox{on}~\partial \Omega_p \\
{\bf u}\cdot {\bf n} & = & z \qquad \mbox{on}~\partial \Omega_u \label{eq1d-pre}
\end{eqnarray}

Here, $\Omega\subset \mathbb{R}^d$, where in practical applications we only consider two- or three-dimensional problems, so we set $d=2$ or $3$.  The absolute permeability field $K$,  whose components belong to $L^\infty(\Omega)$, is a symmetric, uniformly positive definite tensor. The source term $f$ belongs to $L^2(\Omega)$, while the prescribed pressure and normal velocity boundary conditions are given by $g\,\in\,H^{\frac12}(\partial\Omega_p)$ and $z\,\in\,H^{-\frac12}(\partial\Omega_u)$, respectively. The outer unit normal vector to $\partial{\Omega}$ is set as ${\bf n}$.

\subsection{The MRCM with Oversampling  (MRCM-O)}

The approach utilized in this study is founded on the MRCM-O \cite{MRCM-OS, zhou2025, i-MRCM-OS} which is briefly summarized in this section. Let $\mathcal{T}_h$  denote a Cartesian mesh that subdivides $\Omega\subset \mathbb{R}^d$ into $d$-dimensional rectangular elements. From this discretization, $\Omega$ is subdivided into a set of non-overlapping subdomains $\{\Omega_i\}_{i=1,\ldots,m}$,  and let $\Gamma$ be the union of all interfaces $\Gamma_{i,j}=\overline{\Omega}_i\cap\overline{\Omega}_j, \forall i,j=1\dots\,m$, and define $\Gamma_i = \partial\Omega_i\setminus\partial\Omega$. For each subdomain $\Omega_i$, define $\hat{\Omega}_i$ as the augmented subdomain comprising $\Omega_i$ and an adjoining region. These definitions are illustrated in Figure \ref{overlapjpg-pre}. The variational formulation of the MRCM-O, expressed in discrete form, is: 
For $i=1,\ldots,m$, compute $({\bf u}_h^i,p_h^i,\lambda_h^i)\,\in\,{\bf V}_{hz}^i\times Q_h^i\times \Lambda^i_{H}$  as follows:

\begin{eqnarray}
  (K^{-1}{\bf u}_h^i,{\bf v})_{\Omega_i}-(p_h^i,\nabla\cdot {\bf v})_{\Omega_i} 
  +(\beta_i\,{\bf u}_h^i\cdot\check{\bf n}^i,{\bf v}\cdot\check{\bf n}^i)_{\Gamma_i}
  +(\lambda_h^i,{\bf v}\cdot\check{\bf n}^i)_{\Gamma_i} \nonumber \\ = -(g,{\bf v}\cdot\check{\bf n}^i)_{\partial\Omega_i\cap\partial\Omega_p},  \label{eq11d-pre}\\
  (q,\nabla\cdot {\bf u}_h^i)_{\Omega_i}  =   (f,q)_{\Omega_i}, \label{eq12d-pre}\\
  \sum_{i=1}^m ({\bf u}_h^i\cdot \check{\bf n}^i,M)_{\Gamma_i} = 0,  \label{eq13d-pre}\\
  \sum_{i=1}^m (\beta_i\,{\bf u}_h^i\cdot \check{\bf n}^i+\lambda_h^i,V \,\check{\bf n}^i\cdot\check{\bf n})_{\Gamma_i} = 0,  \label{eq14d-pre}
\end{eqnarray}
hold for all $({\bf v},q)\,\in\,{\bf V}_{h0}^i$ and for all $(M,V)\,\in\, M_H \times V_H \subset F_h(\mathcal{E}_h)\times\,F_h(\mathcal{E}_h)$, where
\begin{eqnarray*}
  {\bf V}_{hy}^i&=&\{{\bf v}\,\in\,{\bf V}_{h}^i~,~
  {\bf v}\cdot \check{\bf n}=y~\mbox{on}\,\partial\Omega_i\cap\partial\Omega_u
  \}~, \label{eq:Vhy-pre}
\end{eqnarray*}
and ${\bf V}_{h}^i\times Q_h^i$ are the lowest-order Raviart-Thomas \cite{RaviartThomas::1977}
spaces for velocity and pressure defined for $\Omega_i$, and
\begin{equation*}
  F_h(S_h) = \{ f:{S}_h\to \mathbb{R}~|~f|_e\,\in\,\mathbb{P}_0~,
  ~\forall\,e\,\in\,{S}_h \}
\end{equation*}
where $\mathcal{E}_h$ denotes the set of all faces (or edges) of $\mathcal{T}_h$ that lie on $\Gamma$, and ${S}_h$ is any subset of edges (or faces) of $\mathcal{T}_h$ (see \cite{MRCM-OS, zhou2025} for further details).

Following \cite{guiraldello2018multiscale}, we note that when imposing Robin boundary conditions 
\begin{equation}
 -\beta_i \, {\bf u}^i \cdot \check{\bf n}^i + p^i = g_R \label{eq:genRobin-pre}
\end{equation}
where  $g_R$  is a specified value, we set
\begin{equation}
 \beta_i\left({\bf x}\right) = \dfrac{\alpha H}{K_{H}\left({\bf x}\right)}
 \end{equation}
where the harmonic average of the neighboring $K$ values is denoted by $K_{H}$. The parameter $\alpha$, which governs the balance between the normal component of the flux and the pressure in the Robin boundary condition, is dimensionless. Section \ref{intro-pre} provides a detailed discussion on the importance of $\alpha$.

The Lagrange multiplier spaces $\Lambda^i_H = \mbox{span}\left\{\phi_i^1,\phi_i^2,..,\phi_i^N\right\} \subset F_h(\mathcal{E}_h\cap\Gamma_i)$ with $\phi_i^k$ given by
\begin{equation*}
  \phi_i^k = -\beta_i\,{\bf u}_h^{k}\cdot {\bf \check{n}}^i|_{\Gamma_i} + \pi^k|_{\Gamma_i}, \ \forall\,k=1,\dots,N,  
  \label{informed_space-pre}  
\end{equation*}
where ${\bf u}_h^{k}\cdot {\bf \check{n}}^i$ and $\pi^k$ are, respectively, the normal component of the velocity field
and pressure, retrieved on the interface $\Gamma_i$ as solutions of the discrete
problem given by

\begin{equation}
  \begin{array}{rclll}
    {\bf u}_h^{k} &=& -K\nabla_h\,p_h^{k} &&\mbox{in} \ \hat\Omega_i \\
    \nabla_h\cdot{\bf u}_h^{k} &=& 0 &&\mbox{in}  \ \hat\Omega_i \\
    p_h^{k} & = & 0 \qquad &&\mbox{on}~\partial\hat\Omega_i\cap\partial\Omega_p \\
    {\bf u}_h^{k}\cdot {\bf n}^i & = & 0 \qquad &&\mbox{on}~\partial\hat\Omega_i\cap\partial\Omega_u \\
    -\beta_i\,{\bf u}_h^{k}\cdot {\bf n}^i + p_h^{k} &=& \lambda^k &&\mbox{on}~\partial\hat\Omega_i\setminus\partial\Omega
  \end{array}. \label{eq:oversampling_problem-pre}
\end{equation}

To efficiently solve the system (\ref{eq11d-pre})–(\ref{eq14d-pre}) and capitalize on the computational advantages offered by multi-core architectures, we employ multiscale basis functions (\emph{MBFs}) (refer to \cite{MRCM-OS, zhou2025, guiraldello2018multiscale, i-MRCM-OS, yotov2009}). After obtaining the solutions to equations (\ref{eq11d-pre})–(\ref{eq14d-pre}), a smoothing procedure is applied for a predetermined number of iterations, as outlined in \cite{MRCM-OS, zhou2025}. 
This comprehensive strategy is known as \emph{MRCM-OS}, and it is implemented through the following algorithm: The initial step involves defining the decomposition of the domain and specifying the oversampling size. Next, the Lagrange multiplier spaces $\Lambda^i_H$ are constructed by solving \eqref{eq:oversampling_problem-pre}. In the third step, the system \eqref{eq11d-pre}--\eqref{eq14d-pre} is solved using MBFs. Finally, a smoothing procedure is applied iteratively for a prescribed number of steps.

The detailed structure of MRCM-OS is presented in \cite{MRCM-OS, zhou2025}, and further details on the weak formulation and its analysis can be found in \cite{MRCM-OS, zhou2025, i-MRCM-OS}.

\begin{figure}[H]
    \centering
    \includegraphics[width = 0.4\textwidth]{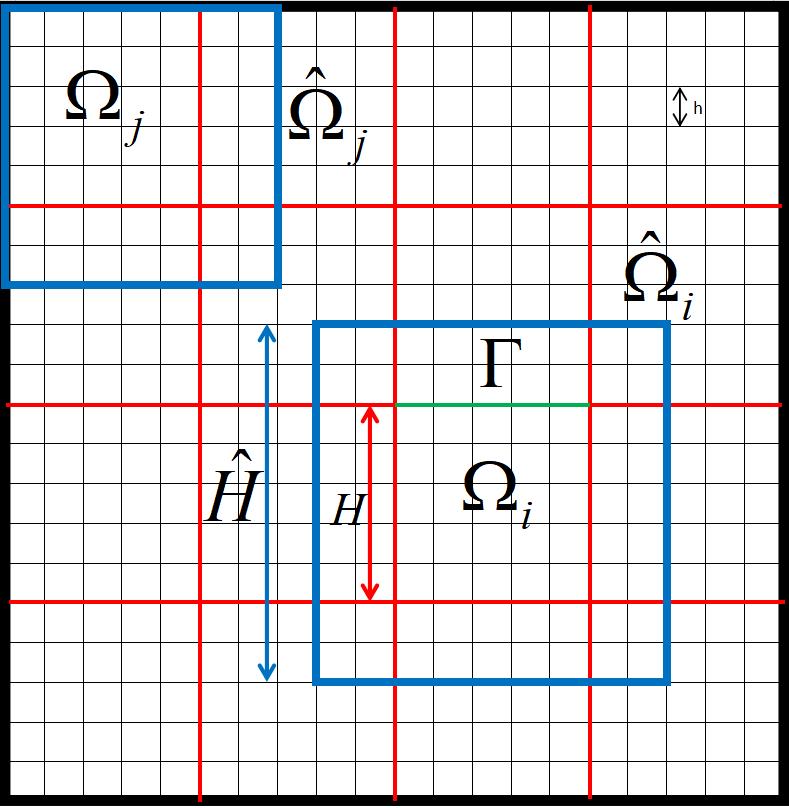}
    \caption{$\Omega_i$ denotes the non-overlapping subdomains, together form the entire domain, with corresponding oversampling regions $\hat{\Omega}_i$. The MRCM-O framework is defined using three numerical scales. These are the fine mesh size $h$, the diameter $H$ of the non-overlapping subdomains, and the larger diameter $\hat H$  of the oversampling regions, which satisfy the inequality $\hat H > H \geq h$.}
    \label{overlapjpg-pre}
\end{figure}

\section{New preconditioner with GMRES}

To derive the linear system from problem (\ref{eq1a-pre})-(\ref{eq1d-pre}), we apply a finite element discretization based on the lowest-order Raviart-Thomas spaces ${\bf V}_{h} \times Q_h$ defined over $\Omega$. This formulation after static condensation results in the following algebraic system:  
\begin{equation}  
  A p = b,  
\end{equation}  
where $A$ is the discrete system matrix arising from the discretized differential operators, $p$ denotes the discrete pressure unknowns at mesh points, and $b$ is the right-hand side vector incorporating the contributions from source terms and boundary conditions.
The matrix $A$ is typically large and sparse, requiring the use of efficient iterative solvers and preconditioning strategies to accelerate convergence \cite{FGMRES, ALI2024112609}. The permeability tensor $K$, which can exhibit significant contrast and heterogeneity, particularly in the SPE10 benchmark with channel structures considered in this study, introduces large variations that lead to ill-conditioning in the discretized system. These variations pose challenges for iterative solvers, underscoring the need for robust preconditioning techniques to improve the convergence properties of the method \cite{FGMRES, STATHOPOULOS1995197, 10.2514/6.1994-520, 10.1002/nme.3032}.

Although the discrete system matrix A is symmetric positive definite (SPD), the proposed MRCM-OS preconditioning strategy introduces a non-symmetric operator T. In this framework, the preconditioned system no longer satisfies the symmetry requirements of the Preconditioned Conjugate Gradient (PCG) method. Consequently, we employ the Generalized Minimal Residual method (GMRES), which is robust to non-symmetric operators and ensures stable convergence  (Algorithm \ref{fgmres-pre}).

After discretization, the action of the preconditioner can be interpreted as a linear operator $(T: \mathbb{R}^n \to \mathbb{R}^n)$, where $n$ is the number of pressure unknowns. Given a residual vector $r \in \mathbb{R}^n$, the application $Tr$ is defined by solving the local subdomain problems with source terms induced by $r$, followed by the multiscale reconstruction procedure.
This defines an implicit, matrix-free linear operator that approximates the inverse of the global stiffness matrix.

\begin{algorithm}[h]
  \caption{Generalized Minimum Residual Method (GMRES)}
  \begin{algorithmic}[1]
    \State Compute $r_0 = b - A p_0$, $\beta = \|r_0\|$, and set $v_1 = r_0 / \beta$
    \For{$j = 1, \dots, m$}
    \State Apply $T$ to $v_j$: $z_j = T(v_j)$
    \State Compute $w = A z_j$
    \For{$i = 1, \dots, j$}
    \State $h_{i,j} = (w, v_i)$
    \State $w = w - h_{i,j} v_i$
    \EndFor
    \State Compute $h_{j+1,j} = \|w\|$ and $v_{j+1} = w / h_{j+1,j}$
    \EndFor
    \State Define $Z_m = [z_1, \dots, z_m]$ and $\bar{H}_m = \{h_{i,j}\}_{1 \leq i \leq j+1, 1 \leq j \leq m}$
    \State Compute $y_m = \arg\min_y \| \beta e_1 - \bar{H}_m y \|$
    \State Compute $p_m = p_0 + Z_m y_m$
    \If{stopping criterion is met}
    \State Stop
    \Else
    \State Set $p_0 \leftarrow p_m$ and go to step 1
    \EndIf
  \end{algorithmic}\label{fgmres-pre}
\end{algorithm}

 In order to employ MRCM-OS as preconditioners, a minor modification to the equations (\ref{eq11d-pre})-(\ref{eq14d-pre}) is required. The modified MRCM-O is given as follows: Find $({\bf u}_h^i,p_h^i,\lambda_h^i)\,\in\,{\bf V}_{h0}^i\times Q_h^i
\times \Lambda^i_{H}$, for
$i=1,\ldots,m$, such that
\begin{eqnarray*}
  (K^{-1}{\bf u}_h^i,{\bf v})_{\Omega_i}-(p_h^i,\nabla\cdot {\bf v})_{\Omega_i} 
  +(\beta_i\,{\bf u}_h^i\cdot\check{\bf n}^i,{\bf v}\cdot\check{\bf n}^i)_{\Gamma_i}
  +(\lambda_h^i,{\bf v}\cdot\check{\bf n}^i)_{\Gamma_i} = 0, \\
  (q,\nabla\cdot {\bf u}_h^i)_{\Omega_i}  =   (r,q)_{\Omega_i},\\
  \sum_{i=1}^m ({\bf u}_h^i\cdot \check{\bf n}^i,M)_{\Gamma_i} = 0, \\
  \sum_{i=1}^m (\beta_i\,{\bf u}_h^i\cdot \check{\bf n}^i+\lambda_h^i,V \,\check{\bf n}^i\cdot\check{\bf n})_{\Gamma_i} = 0,
\end{eqnarray*}
hold for all $({\bf v},q)\,\in\,{\bf V}_{h0}^i$ and for all $(M,V)\,\in\, M_H \times V_H \subset F_h(\mathcal{E}_h)\times\,F_h(\mathcal{E}_h)$. Notice that, the MRCM-OS is now defined with homogeneous boundary conditions and the source term $f=r$, where $r$ denotes the current residual.

The Preconditioning operator T is implemented as specialized matrix-free function within the MRCM-OS framework, rather than as an explicit matrix. By directly incorporating problem-specific information, this operator is specifically engineered to mitigate the severe ill-conditioning induced by high-contrast permeability $K$. This tight integration of the MRCM-OS strategy ensures efficient application within each GMRES iteration, significantly reducing the number of iterations required for convergence while maintaining solver robustness to effectively address complex heterogeneities, such as those present in the SPE10 benchmark.

\section{Numerical Studies}\label{Numerical part}

All numerical simulations were performed on the Santos Dumont high-performance computing (HPC) cluster at the LNCC in Brazil. We consider a sequence of test problems of increasing difficulty for linear solvers: (i) two problems with analytical solutions: one with constant permeability, slightly modified from \cite{MRCM-OS, zhou2025}, and another with a spatially varying permeability field; (ii) three highly heterogeneous permeability fields from the SPE10 dataset \cite{SPE10} (\texttt{http://www.spe.org/web/csp}
); and (iii) three challenging thin high- and low-permeability inclusions cases with strong contrast, 
previously studied in \cite{interface2021}. 

In our studies our proposed preconditioning strategy is compared against the use of the non-overlapping MRCM-based preconditioner (developed in \cite{carvalho2024speeding})
as a reference solver. We refrain from comparing our proposed method against other existing procedures because such comparisons were already carried out in that work.

The computational domains are square regions, defined as $[0,1]\times[0,1]$, for cases (i) and (iii), and a rectangular region for case (ii). In all cases, mixed boundary conditions are imposed: homogeneous Neumann conditions on the top and bottom boundaries, and Dirichlet conditions on the left and right boundaries.

In case (i), the total grid size is $256\times256$. Both problems here use $\cos(2\pi y)$ as the Dirichlet boundary condition on the left and right boundaries. The first problem features a constant permeability $\kappa(x,y)=1$, while the second uses a variable permeability $\kappa(x,y)=1+\sin(\pi x)\sin(\pi y)$. The corresponding source terms are defined as $f(x,y)=8\pi ^2 \cos(2\pi x)\cos(2\pi y)$ for the constant case, and $f(x,y)=2\pi ^2 \cos(\pi x)\cos(\pi y)(1+2\sin(\pi x)\sin(\pi y))$ for the variable case. As a result, the analytical solutions are $p(x,y)=\cos(2\pi x)\cos(2\pi y)$ for the constant permeability case, and $p(x,y)=\cos(\pi x)\cos(\pi y)$ for the variable permeability case. Analytical solutions are available for reference.

In case (ii), the total grid size is $220\times60$, and the pressure is fixed at $1$ on the left boundary and $0$ on the right boundary, with the source term set to $f(x,y)=0$. The computational domain is considered under two partitioning configurations: an $11\times3$ subdomain structure with $H=1/3$, where each subdomain contains a local $20\times20$ grid, and a $22\times 6$ subdomain structure with $H=1/6$, where each subdomain contains a local $10\times10$ grid. The permeability field $\kappa(x,y)$ is selected from three different layers within the SPE10 dataset. Layer 34 features moderate contrast and no high-permeability channel. Layer 40 contains a high-permeability channel, posing greater numerical challenges for the original MRCM. Layer 84 exhibits the highest contrast ratio among all SPE10 layers but lacks a channel structure. 

Case (iii) has the total grid size as $160\times160$, with the domain partitioned into $8\times8$ subdomains, each with a local $20\times20$ grid. The pressure is fixed at $1$ on the left boundary and $0$ on the right boundary, and the source term is $q(x,y)=0$. The permeability field $\kappa(x,y)$ is chosen from three examples studied in  \cite{interface2021}:  high-permeability inclusions, low-permeability inclusions, and a combined configuration. 
The first two have contrast $10^4$, while the combined case has contrast $10^8$.


Table \ref{table:pre:notation} provides the notation utilized in the numerical experiments. For Sections \ref{subsec:analytical} through \ref{subsec:fracture}, we report the $\ell^2$ norm of the algebraic residual $\|b - Ap_m\|_2$. In all subsequent sections, the dimension of the Krylov subspace is fixed at 10.

\begin{table}[htbp]
  \centering
  \begin{tabular}{|l|l|}
    \hline
    \textbf{Notation} & \textbf{Description} \\
    \hline
    \textit{MRCM} & Original MRCM method \\
    \textit{$lh, kS$} & Solution with oversampling size $lh$, followed by $k$ smoothing steps\\
    \textit{$\#$} & Number of iterations required to reach the tolerance\\
    \hline
  \end{tabular}
  \caption{Notation used in numerical experiments.}
  \label{table:pre:notation}
\end{table}

Section \ref{subsec:analytical} presents the numerical studies for the problems with analytical solution. Section \ref{subsec:SPE10} reports the numerical results for the SPE10 benchmark, while Section \ref{subsec:fracture} focuses on high-contrast inclusions problems. For all these cases, we investigate how different subdomain partitioning, oversampling sizes and varying numbers of smoothing steps influence the performance of the method.

Based on the results obtained from MRCM-OS \cite{MRCM-OS, zhou2025}, all subsequent analyses will be conducted using the $\alpha$ parameter set to 10, as this yields the best outcomes. The tolerance of the residual is set as $10^{-8}$.

\subsection{Problems with analytical solution}\label{subsec:analytical}

We report results for different oversampling sizes and numbers of smoothing steps, across four subdomain partition configurations: $2\times 2$, $4\times 4$, $8\times 8$, and $16\times 16$. The MRCM-OS configuration includes oversampling sizes of $2h$ and $4h$, as well as 2 and 4 smoothing steps. Tables \ref{table:num:analytical constant} and \ref{table:num:analytical variable} present the iteration counts required for both homogeneous and variable permeability problems to converge under various settings, as well as their final residuals after meeting the tolerance.

\begin{table}[htbp]
  \centering
  \begin{tabular}{|c|c|c|c|c|c|c|c|c|}
    \hline
    \textbf{} & \multicolumn{8}{|c|}{\textbf{Homogeneous permeability field problem}} \\
    \hline
    \textbf{Partition} & \multicolumn{2}{|c|}{$2\times 2$} & \multicolumn{2}{|c|}{$4\times 4$} & \multicolumn{2}{|c|}{$8\times 8$} & \multicolumn{2}{|c|}{$16\times 16$} \\
    \hline
    \diagbox[width=6em,height=1.5em]{\textbf{}} & {$\#$} & {Residual}  & {$\#$} & {Residual} & {$\#$} & {Residual} & {$\#$} & {Residual} \\
    \hline
    \textit{MRCM} & 9 & $2.3\times10^{-9}$ & 11 & $6.1\times10^{-9}$  & 8 & $3.9\times10^{-9}$ & 5 & $3.7\times10^{-9}$  \\
    \hline
    \textit{$2h$} & 2 & $8.2\times10^{-12}$ & 2 & $4.0\times10^{-10}$ & 2 & $8.2\times10^{-12}$ & 2 & $8.2\times10^{-12}$  \\
    \hline
    \textit{$4h$}  & 2 & $8.2\times10^{-12}$  & 2 & $8.4\times10^{-12}$ & 2 & $8.2\times10^{-12}$ & 2 & $8.2\times10^{-12}$  \\
    \hline
    \textit{$2h,2S$} & 1 & $7.3\times10^{-9}$ &  2 & $8.2\times10^{-12}$ & 1 & $2.2\times10^{-10}$ & 1 & $1.4\times10^{-10}$  \\
    \hline
    \textit{$2h,4S$} & 1 & $1.5\times10^{-10}$ &  1 & $4.0\times10^{-10}$ & 1 & $1.5\times10^{-10}$ & 1 & $1.5\times10^{-10}$  \\
    \hline
    \textit{$4h,2S$} & 1 & $1.6\times10^{-10}$ &  1 & $5.9\times10^{-10}$ & 1 & $1.5\times10^{-10}$ & 1 & $1.5\times10^{-10}$  \\
    \hline
    \textit{$4h,4S$} & 1 & $1.5\times10^{-10}$ &  1 & $1.5\times10^{-10}$ & 1 & $1.5\times10^{-10}$ & 1 & $1.5\times10^{-10}$  \\
    \hline
  \end{tabular}
  \caption{Number of iterations and corresponding final residuals for a constant permeability problem with analytical solution.}
  \label{table:num:analytical constant}
\end{table}

\begin{table}[htbp]
  \centering
  \begin{tabular}{|c|c|c|c|c|c|c|c|c|}
    \hline
    \textbf{} & \multicolumn{8}{|c|}{\textbf{Variable permeability field problem}} \\
    \hline
    \textbf{Partition} & \multicolumn{2}{|c|}{$2\times 2$} & \multicolumn{2}{|c|}{$4\times 4$} & \multicolumn{2}{|c|}{$8\times 8$} & \multicolumn{2}{|c|}{$16\times 16$} \\
    \hline
     \diagbox[width=6em,height=1.5em]{\textbf{}} & {$\#$} & {Residual}  & {$\#$} & {Residual} & {$\#$} & {Residual} & {$\#$} & {Residual} \\
    \hline
     \textit{MRCM} & 6 & $7.9\times10^{-10}$ & 10 & $8.4\times10^{-9}$  & 7 & $6.8\times10^{-9}$ & 5 & $2.0\times10^{-9}$\\
    \hline
    \textit{$2h$}& 2 & $1.4\times10^{-11}$ & 2 & $1.1\times10^{-10}$  & 2 & $1.4\times10^{-11}$ & 2 & $1.4\times10^{-11}$\\
    \hline
    \textit{$4h$} & 2 & $1.4\times10^{-9}$ & 2 & $1.4\times10^{-11}$  & 2 & $1.4\times10^{-11}$ & 1 & $7.1\times10^{-9}$\\
    \hline
    \textit{$2h,2S$}& 2 & $1.2\times10^{-9}$ & 2 & $1.4\times10^{-11}$  & 1 & $9.5\times10^{-11}$ & 1 & $9.5\times10^{-11}$\\
    \hline
    \textit{$2h,4S$} & 1 & $8.9\times10^{-11}$ & 1 & $2.5\times10^{-10}$  & 1 & $1.0\times10^{-10}$ & 1 & $8.8\times10^{-11}$\\
    \hline
    \textit{$4h,2S$} & 1 & $9.4\times10^{-11}$ & 1 & $3.0\times10^{-10}$  & 1 & $8.0\times10^{-11}$ & 1 & $8.6\times10^{-11}$\\
    \hline
    \textit{$4h,4S$}  & 1 & $9.1\times10^{-11}$ & 1 & $1.2\times10^{-10}$  & 1 & $9.4\times10^{-11}$ & 1 & $8.4\times10^{-11}$\\
    \hline
  \end{tabular}
  \caption{Number of iterations and corresponding final residuals for the problem with a variable permeability field with analytical solution.}
  \label{table:num:analytical variable}
\end{table}

Since a smaller number of GMRES iterations indicates faster convergence, Figs. \ref{22} and \ref{44} illustrate the residual decay with respect to the number of GMRES iterations for different parameter settings in the $2\times 2$ and $4\times 4$ partitions, for both constant and variable permeability fields, providing a clear picture of how the residual decreases across these settings. Other configurations ($8\times 8$ and $16\times 16$ partitions) reach the tolerance even more quickly than these two cases.

\begin{figure}[H]
    \centering
    \includegraphics[width = 0.48\textwidth]{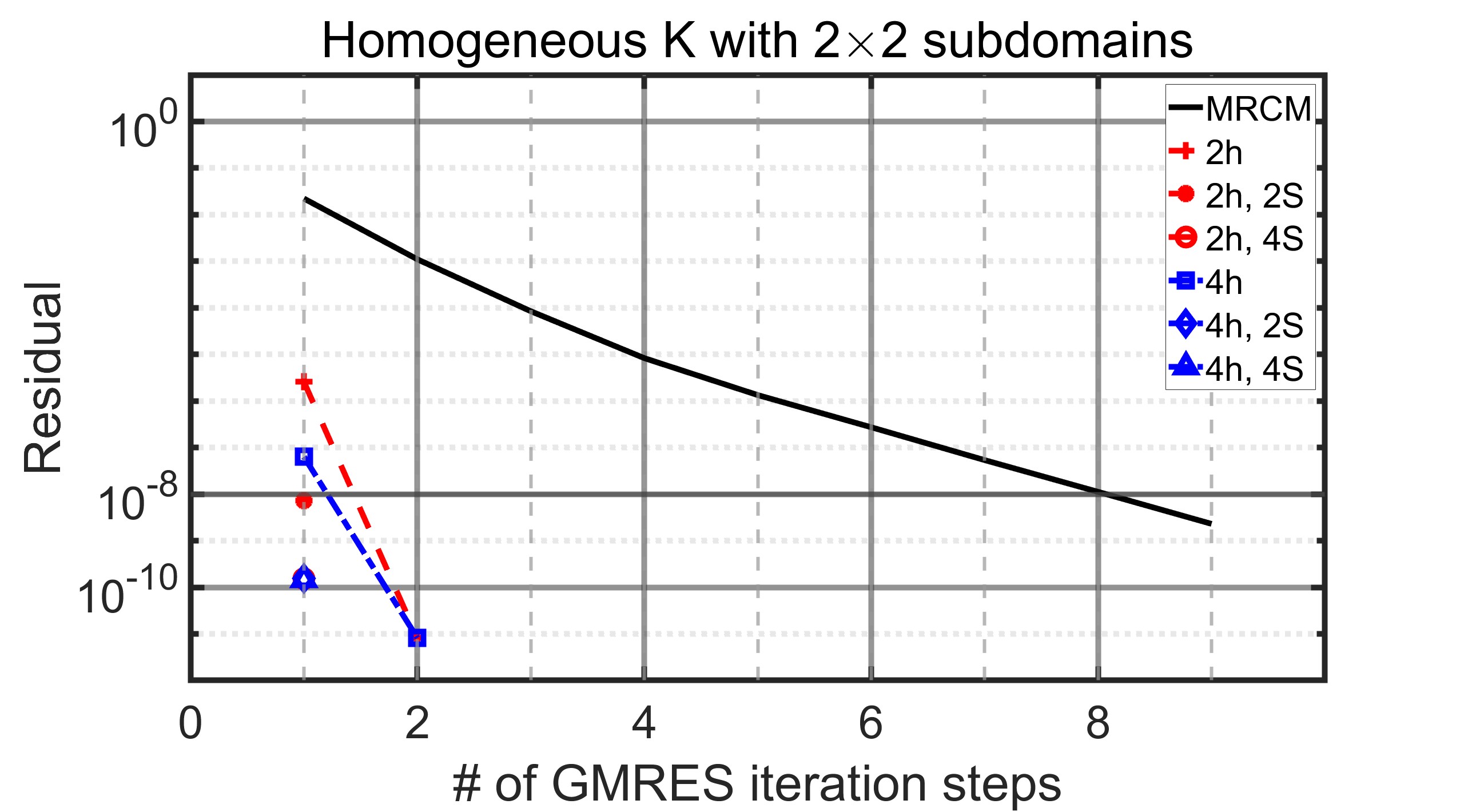}
    \includegraphics[width = 0.48\textwidth]{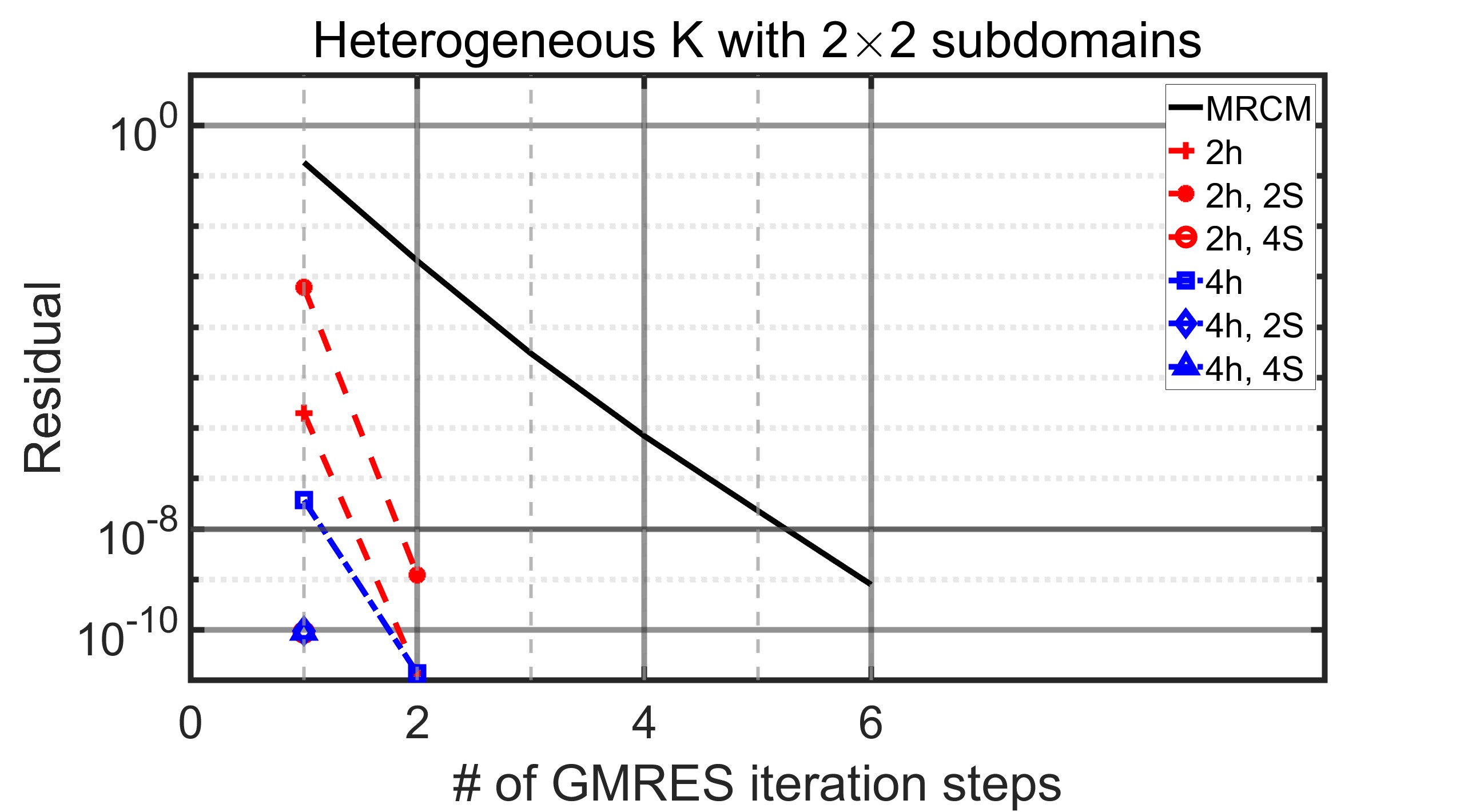}
    \caption{Number of iterations: $2\times 2$ partition for the constant permeability problem (left) and the variable permeability problem (right).}
    \label{22}
\end{figure}

\begin{figure}[H]
    \centering
    \includegraphics[width = 0.48\textwidth]{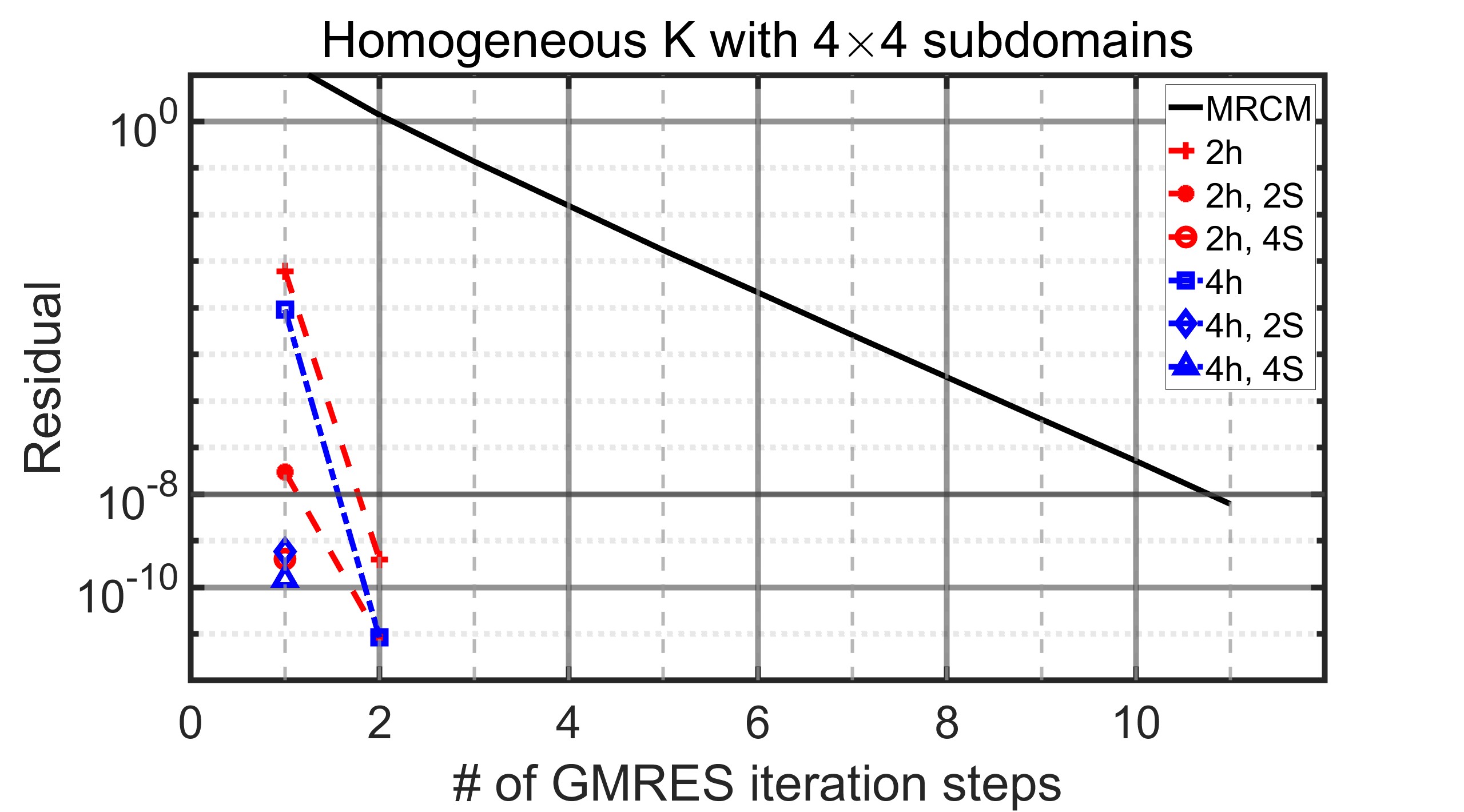}
    \includegraphics[width = 0.48\textwidth]{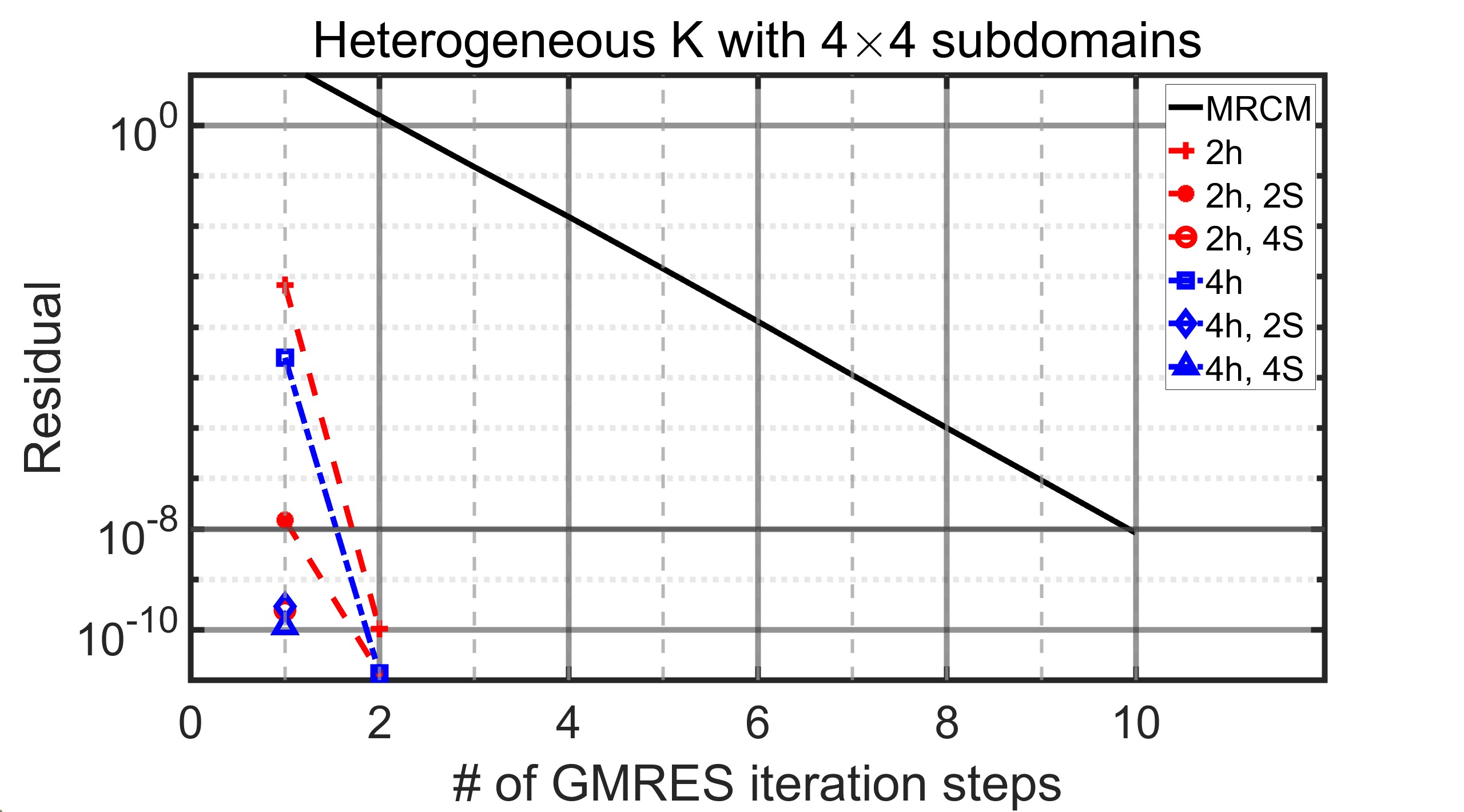}
     \caption{Number of iterations: $4\times 4$ partition for the constant permeability problem (left) and the variable permeability problem (right).}
    \label{44}
\end{figure}

We observe that, for both constant and variable permeability fields, the use of the original MRCM as the reference preconditioner for GMRES consistently requires more than five iterations to reach the desired tolerance. In contrast, employing MRCM-O (MRCM-OS without the smoothing steps) reduces the iteration count to just two. Furthermore, when smoothing procedures are applied, the method converges even faster, requiring only a single iteration for most situations.

We note that when two results have the same number of iterations, the one with smaller final residual is considered superior. These results also highlight that, in the partition-refinement study, increasing the number of subdomains improves performance when the oversampling size and the number of smoothing steps are held constant.

\subsection{Layers of the SPE10 dataset}\label{subsec:SPE10}

In this section, we focus on heterogeneous problems for three different layers from the SPE10 project. Figs. \ref{SPE10perm} shows high-contrast permeability fields of the layers we consider in our studies on a $log_{10}$ scale.

\begin{figure}[H]
    \centering
    \includegraphics[width = 1.1\textwidth]{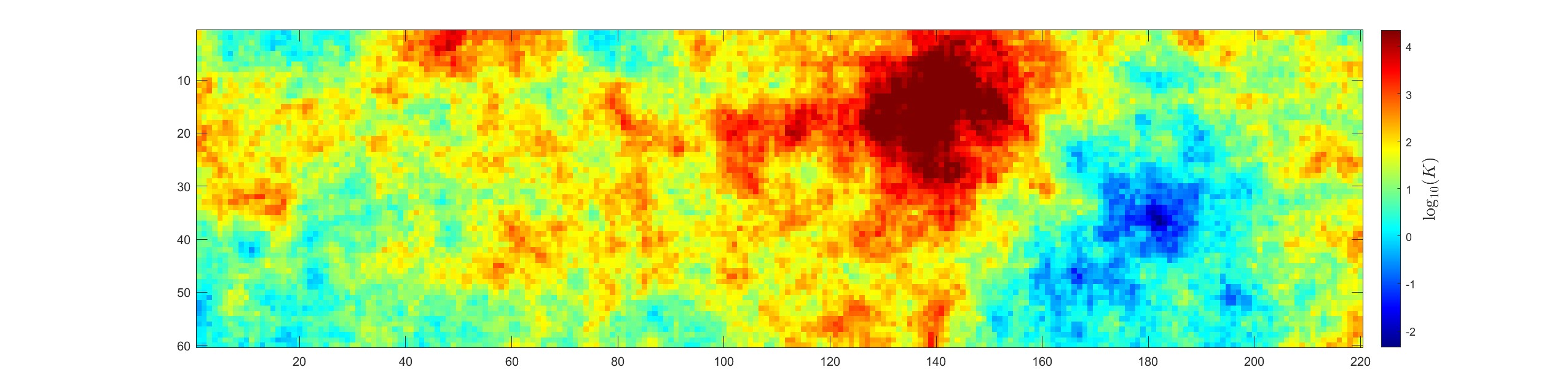}
    \includegraphics[width = 1.1\textwidth]{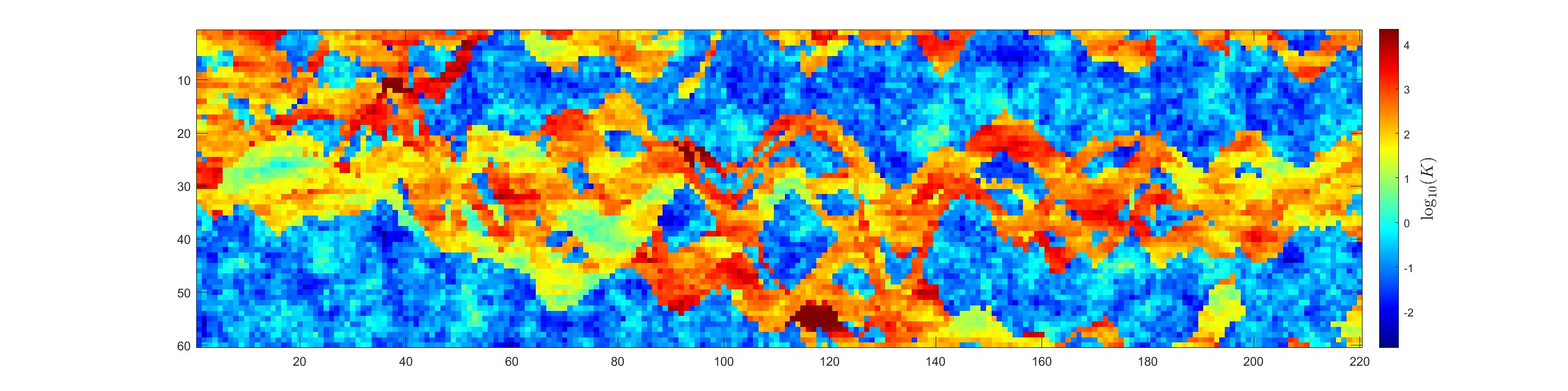}
    \includegraphics[width = 1.1\textwidth]{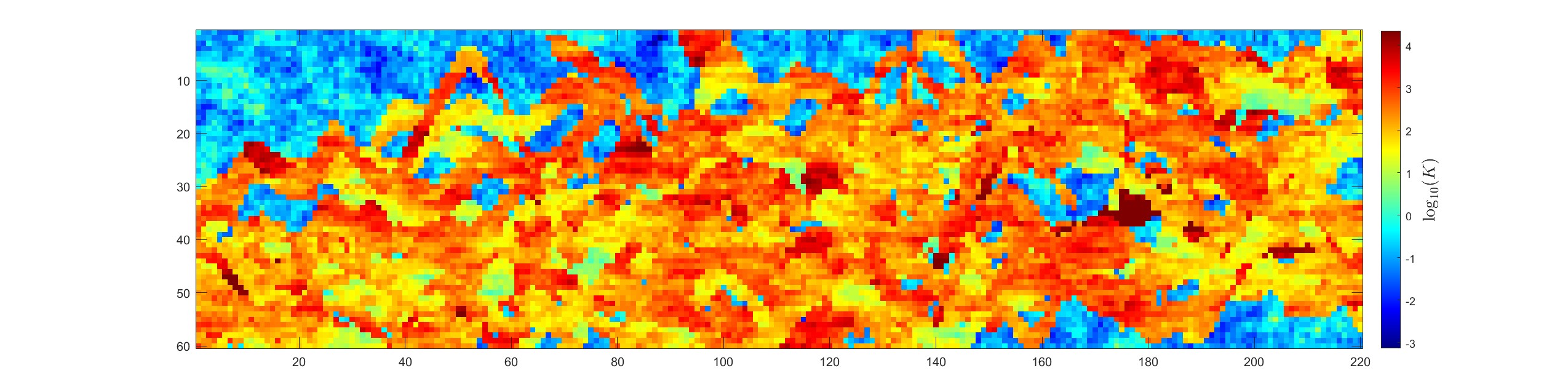}
    \caption{Log10 of permeability fields: Layer 34 (top), 40 (middle), 84 (bottom).}
    \label{SPE10perm}
\end{figure}

Similar to Section~\ref{subsec:analytical}, Tables \ref{table:num:SPE10 11,3} and \ref{table:num:SPE10 22,6} show the number of iterations required for the three selected layers to reach the specified tolerance using the same MRCM-OS configuration. The subdomain refinement study is conducted for two partitioning strategies: $11\times 3$ and $22\times 6$.

\begin{table}[htbp]
  \centering
  {SPE10: Domain Decomposition: $11\times3$}

  \begin{tabular}{|c|c|c|c|c|c|c|}
    \hline
    \textbf{Layer} & \multicolumn{2}{|c|}{\textbf{34}} & \multicolumn{2}{|c|}{\textbf{40}}  & \multicolumn{2}{|c|}{\textbf{84}}\\
    \hline
    \diagbox[width=4.5em,height=1.5em]{\textbf{}} & {$\#$} & {Residual}  & {$\#$} & {Residual} & {$\#$} & {Residual} \\
    \hline
    \textit{MRCM} & 9 & $5.8\times10^{-9}$ & 30+ & $1.2\times10^{-1}$ & 18 & $7.2\times10^{-9}$\\
    \hline
    \textit{$2h$} & 2 & $5.9\times10^{-9}$  & 12 & $7.5\times10^{-9}$ & 3 & $1.9\times10^{-9}$\\
    \hline
    \textit{$4h$}  & 2 & $5.8\times10^{-9}$  & 4 & $2.0\times10^{-9}$ & 2 & $2.6\times10^{-9}$\\
    \hline
    \textit{$2h,2S$} & 2 & $5.7\times10^{-9}$ & 3 & $2.0\times10^{-9}$ & 2 & $1.8\times10^{-9}$\\
    \hline
    \textit{$2h,4S$} & 1 & $9.6\times10^{-9}$ & 2 & $2.0\times10^{-9}$ & 2 & $1.9\times10^{-9}$\\
    \hline
    \textit{$4h,2S$} & 1 &$ 9.0\times10^{-9}$ & 2 & $2.0\times10^{-9}$ & 1 & $2.6\times10^{-9}$\\
    \hline
    \textit{$4h,4S$} & 1 & $7.8\times10^{-9}$ & 2 & $2.1\times10^{-9}$ & 1 & $2.5\times10^{-9}$\\
    \hline
  \end{tabular}
  \caption{Number of iterations and corresponding final residuals for three SPE10 layers for the $11\times3$ subdomain partition.}
  \label{table:num:SPE10 11,3}
\end{table}

\begin{table}[htbp]
  \centering
  {SPE10: Domain Decomposition: $22\times6$}

  \begin{tabular}{|c|c|c|c|c|c|c|}
    \hline
    \textbf{Layer} & \multicolumn{2}{|c|}{\textbf{34}} & \multicolumn{2}{|c|}{\textbf{40}}  & \multicolumn{2}{|c|}{\textbf{84}}\\
    \hline
    \diagbox[width=4.5em,height=1.5em]{\textbf{}} & {$\#$} & {Residual}  & {$\#$} & {Residual} & {$\#$} & {Residual} \\
    \hline
    \textit{MRCM}& 6 & $6.0\times10^{-9}$ & 11 & $3.4\times10^{-9}$  &12 & $5.3\times10^{-9}$\\
    \hline
    \textit{$2h$}& 2  & $5.9\times10^{-9}$& 3 & $2.1\times10^{-9}$  & 3 & $1.8\times10^{-9}$\\
    \hline
    \textit{$4h$}& 2 & $5.8\times10^{-9}$ & 2 & $2.0\times10^{-9}$  & 2 & $2.0\times10^{-9}$\\
    \hline
    \textit{$2h,2S$}&2 & $8.4\times10^{-9}$ &  2 & $2.1\times10^{-9}$ & 2 & $1.8\times10^{-9}$\\
    \hline
    \textit{$2h,4S$}&  1 & $8.8\times10^{-9}$  & 1 & $2.7\times10^{-9}$  & 1 & $2.3\times10^{-9}$\\
    \hline
    \textit{$4h,2S$}&  1 & $7.8\times10^{-9}$ & 1 & $2.5\times10^{-9}$  & 1 & $2.4\times10^{-9}$\\
    \hline
    \textit{$4h,4S$}&  1 & $8.6\times10^{-9}$  & 1 & $2.2\times10^{-9}$  & 1 & $2.4\times10^{-9}$\\
    \hline
  \end{tabular}
  \caption{Number of iterations and corresponding final residuals for three SPE10 layers for the $22\times6$ subdomain partition.}
  \label{table:num:SPE10 22,6}
\end{table}

An iteration count exceeding 30 (30+ inside the table) for the original MRCM method applied to layer 40 with the $11\times 3$ subdomain partition indicates that more than 30 iterations are required to reach the specified tolerance. Fig. \ref{SPE10_40} illustrates the residual decay for the high permeability layer, following the same format as in the previous subsection.

\begin{figure}[H]
    \centering
    \includegraphics[width = 0.48\textwidth]{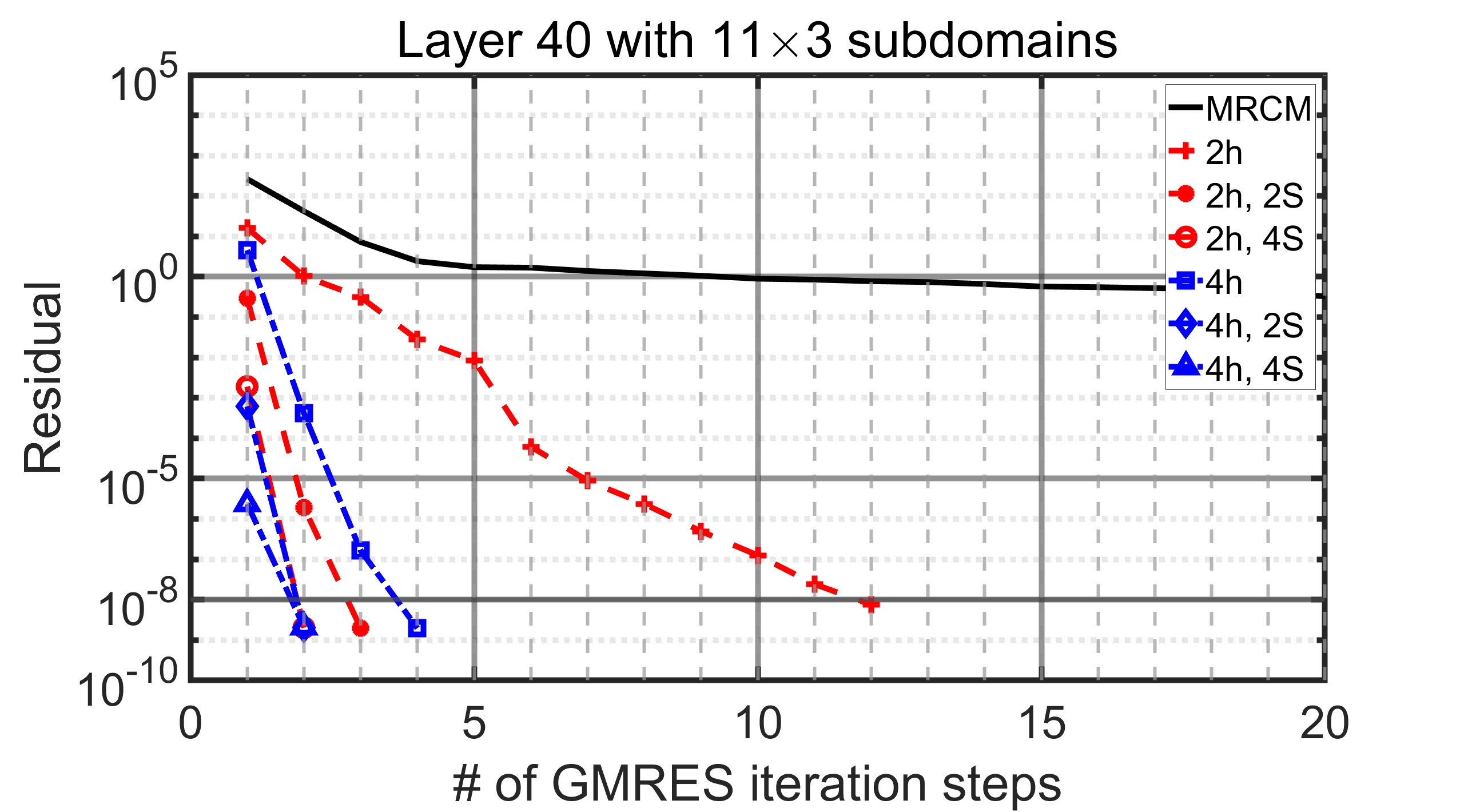}
    \includegraphics[width = 0.48\textwidth]{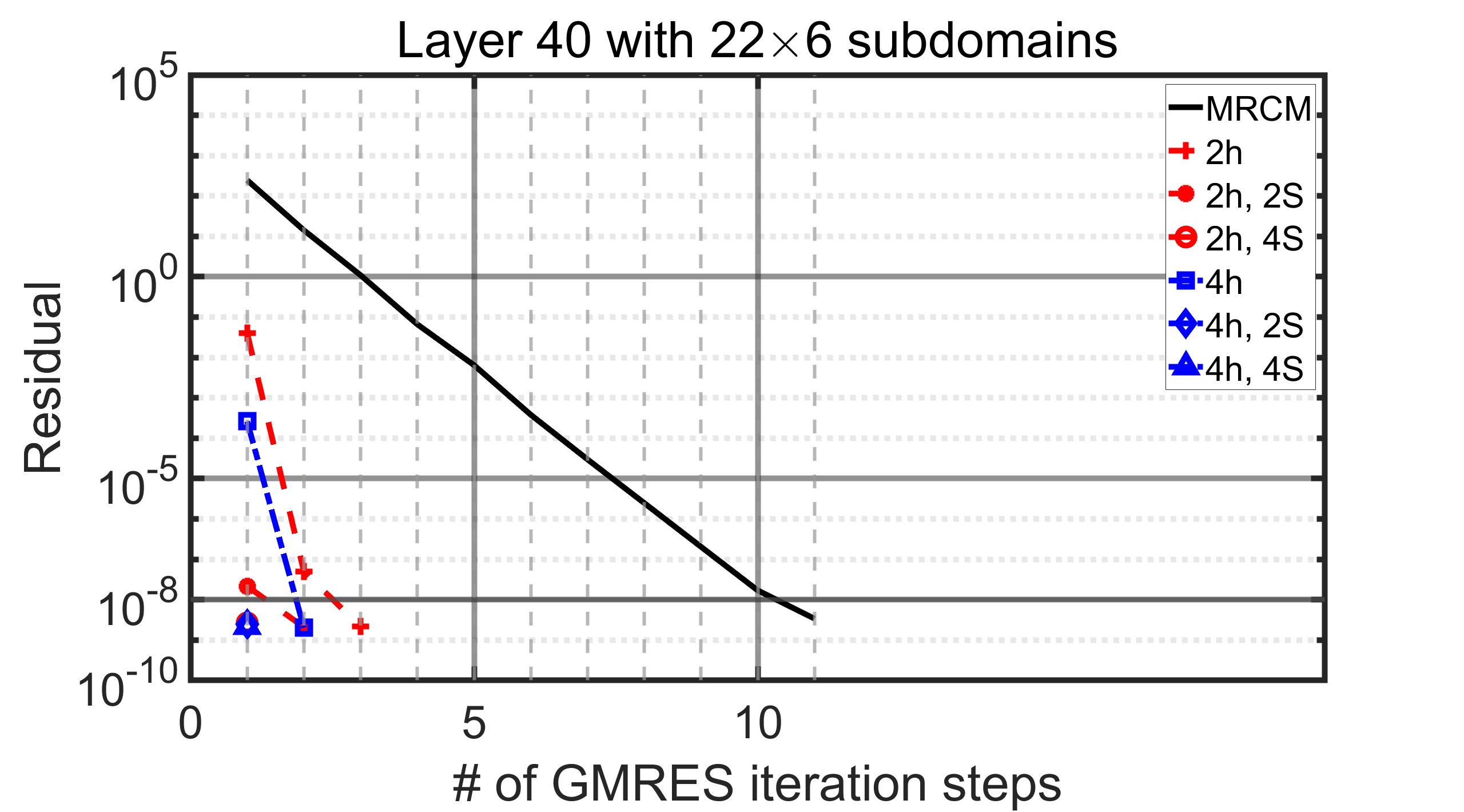}
    \caption{Number of iterations: layer 40 for the SPE10 problem using $11\times 3$ (left) and $22\times 6$ (right) subdomain partitions.}
    \label{SPE10_40}
\end{figure}

From the above results, we observe that layer 34 is the easiest to solve, while layer 40 is the most challenging. This indicates that the presence of a high-permeability channel poses a greater 
challenge for the solver.
Additionally, among the different domain partitions, the $22\times 6$ configuration performs better than the $11\times 3$ configuration, indicating that using more subdomains leads to fewer iterations to reach the tolerance.
 Applying both oversampling and smoothing techniques for the $22\times 6$ partition yields the best performance, typically requiring only one GMRES iteration to reach the desired tolerance. Even for the most difficult case (layer 40) with minimal oversampling and smoothing, convergence is achieved within three iterations. When only oversampling is used, two to three iterations are generally sufficient, even for layer 40. In the partition refinement study, the trend remains consistent with earlier observations: increasing the number of subdomains results in fewer iterations needed to meet the tolerance.

\subsection{Problems with thin high-contrast inclusions}\label{subsec:fracture}
In this section, we consider a set of particularly challenging problems. Three highly difficult cases with high- and -low permeability inclusions, motivated by \cite{interface2021}, are shown in Fig. \ref{fractureperm}. In these cases, the combined (high- and low-permeability inclusions) configuration has a contrast ratio of $10^8$, with the background permeability fixed at 1. The 
permeability fields featuring only high- or low-permeability inclusions have a contrast ratio of $10^4$.

\begin{figure}[H]
    \centering
    \includegraphics[width = 0.49\textwidth]{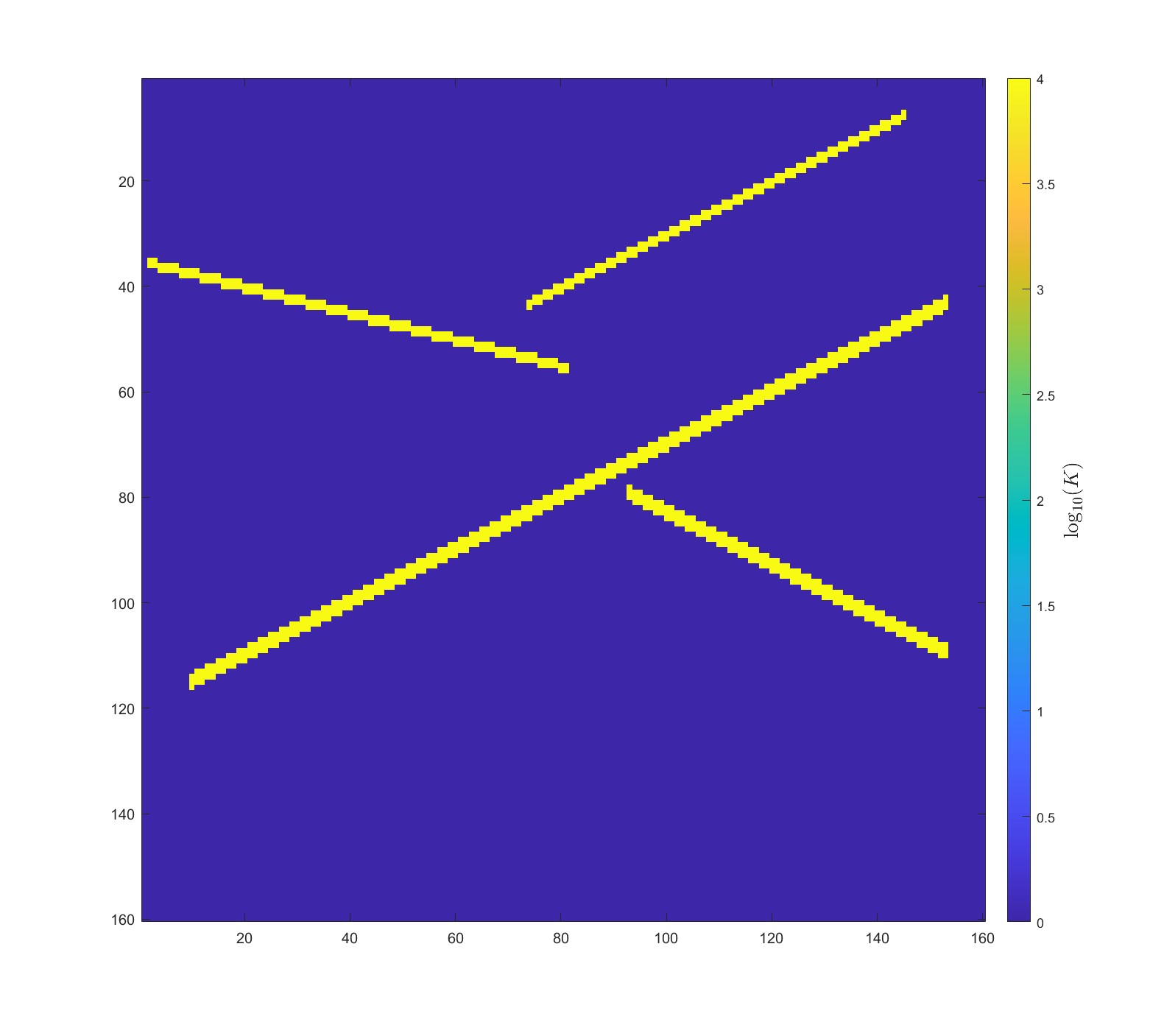}
    \includegraphics[width = 0.49\textwidth]{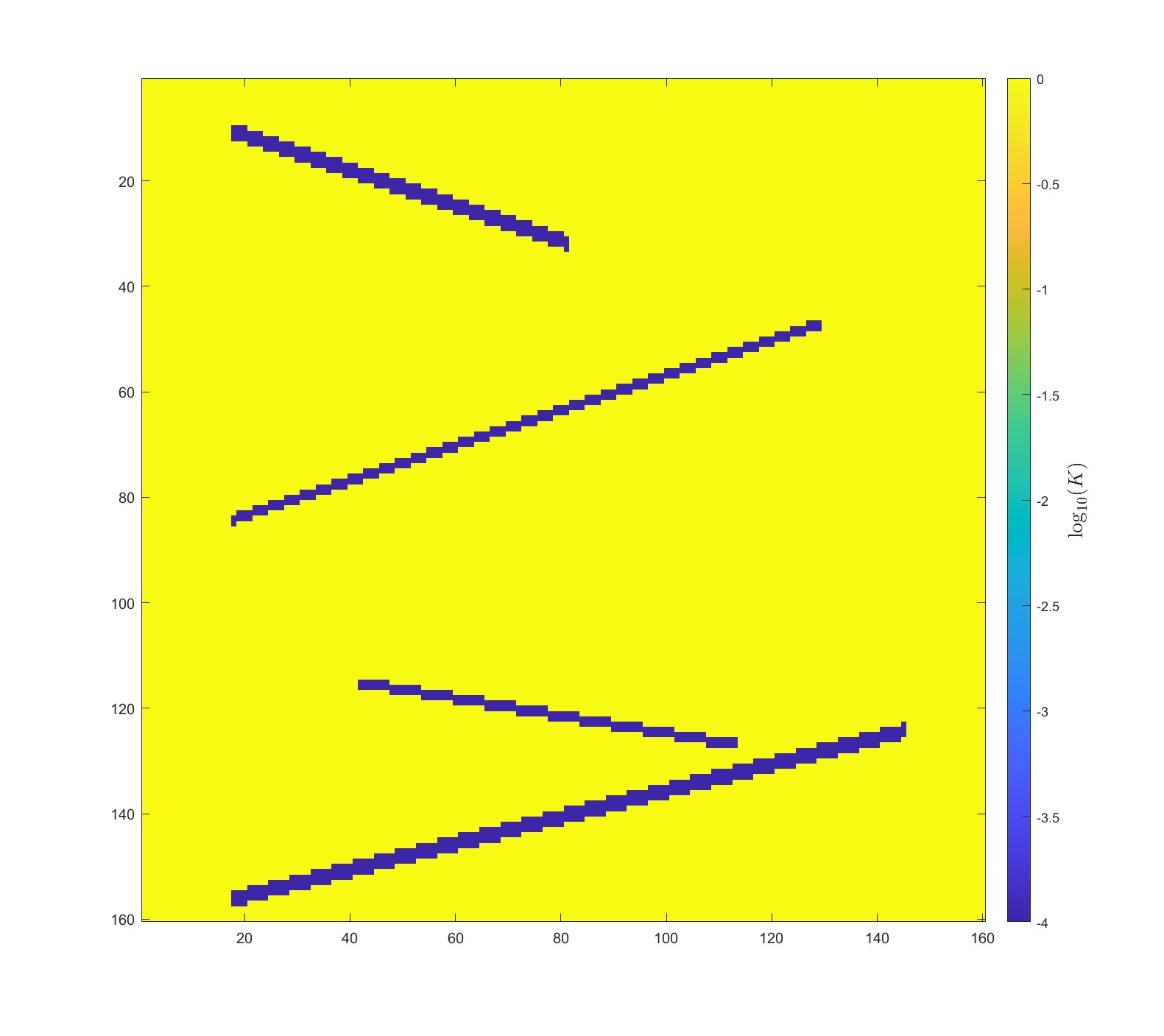}
    \includegraphics[width = 0.49\textwidth]{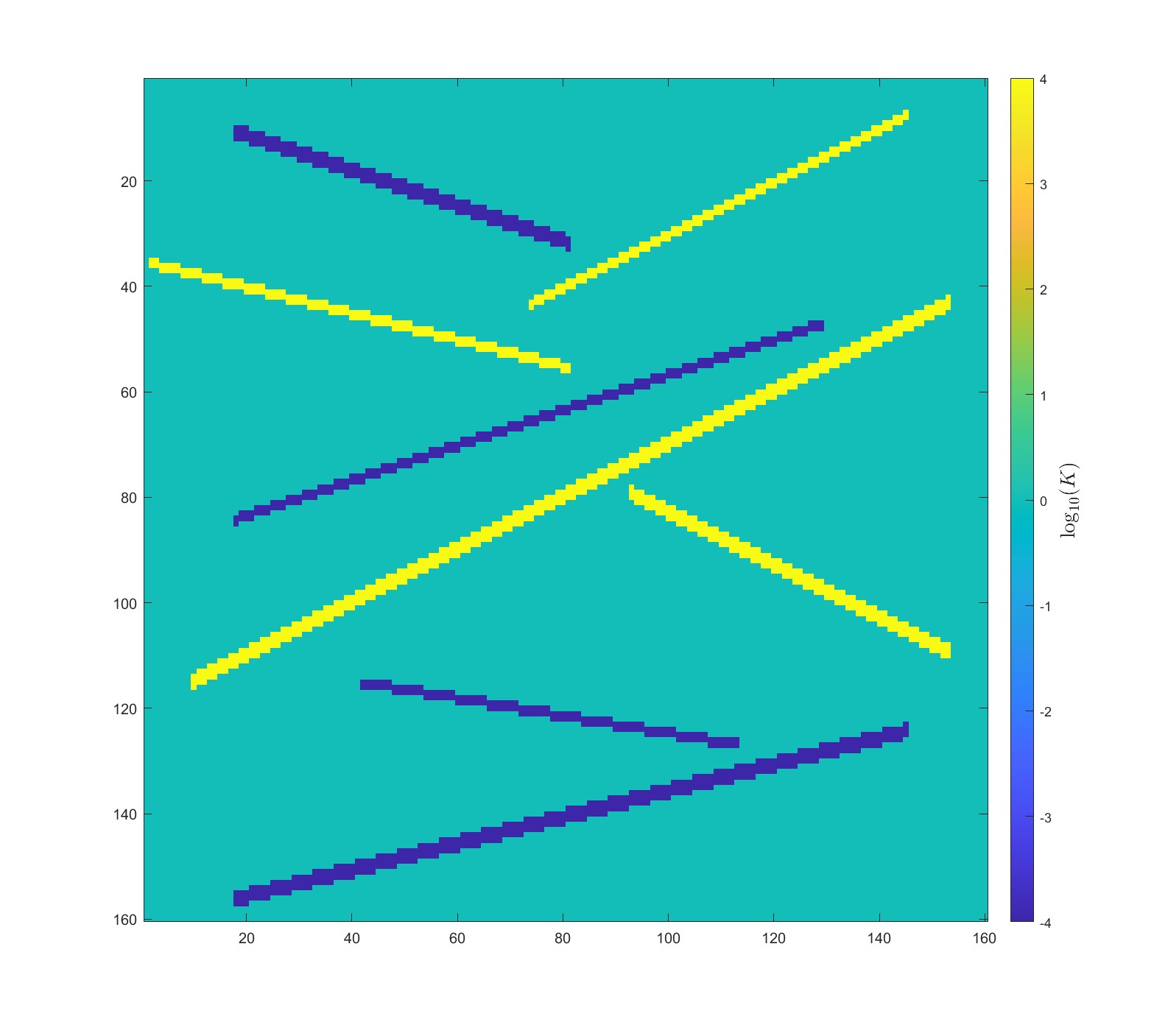}
    \caption{Permeability fields: high-permeability inclusions (left), small-permeability inclusions (right), and a combination of both (bottom).}
    \label{fractureperm}
\end{figure}
Similar to the previous section, Tables~\ref{table:num:fracture1} and \ref{table:num:fracture2} report the number of iterations required to reach convergence (or a prescribed tolerance) for the three selected permeability fields, using the same MRCM-OS configuration.

\begin{table}[htbp]
  \centering
   {High-contrast inclusions: Domain decomposition: $4\times4$}

  \begin{tabular}{|c|c|c|c|c|c|c|}
    \hline
   \textbf{Layer} & \multicolumn{2}{|c|}{\textbf{high-perm. inclusion}} & \multicolumn{2}{|c|}{\textbf{low-perm. inclusion}}  & \multicolumn{2}{|c|}{\textbf{combined inclusions}}\\
    \hline
   \diagbox[width=4.5em,height=1.5em]{\textbf{}} & {$\#$} & {Residual}  & {$\#$} & {Residual} & {$\#$} & {Residual} \\
    \hline
    \textit{MRCM} &30+ &  & 21 & $3.4\times10^{-9}$ & 30+ & \\
    \hline
    \textit{$2h$} & 3 & $8.2\times10^{-9}$  & 3 & $1.5\times10^{-9}$ & 3 & $8.5\times10^{-9}$\\
    \hline
    \textit{$4h$}  & 3 & $7.7\times10^{-9}$  & 3 & $8.2\times10^{-12}$ & 3 & $7.8\times10^{-9}$\\
    \hline
    \textit{$2h,2S$} & 2 & $7.5\times10^{-9}$ &  2 & $6.1\times10^{-12}$ & 2 & $8.2\times10^{-9}$\\
    \hline
    \textit{$2h,4S$} &  2 & $7.5\times10^{-9}$ &  1 & $4.1\times10^{-9}$ & 2 & $7.6\times10^{-9}$\\
    \hline
    \textit{$4h,2S$} &  2 & $7.9\times10^{-9}$ &  1 & $9.1\times10^{-9}$ & 2 & $7.8\times10^{-9}$\\
    \hline
    \textit{$4h,4S$} & 2 &$7.4\times10^{-9}$ &  1 & $8.8\times10^{-12}$ & 2 & $7.3\times10^{-9}$\\
    \hline
  \end{tabular}
  \caption{Number of iterations and corresponding final residuals for the problems with high-permeability inclusions, low permeability inclusions and both inclusions.}
  \label{table:num:fracture1}
\end{table}

\begin{table}[htbp]
  \centering
   {High-contrast inclusions: Domain decomposition: $8\times8$}

  \begin{tabular}{|c|c|c|c|c|c|c|}
    \hline
   \textbf{Layer} & \multicolumn{2}{|c|}{\textbf{high-perm. inclusion}} & \multicolumn{2}{|c|}{\textbf{low-perm. inclusion}}  & \multicolumn{2}{|c|}{\textbf{combined inclusions}}\\
    \hline
   \diagbox[width=4.5em,height=1.5em]{\textbf{}} & {$\#$} & {Residual}  & {$\#$} & {Residual} & {$\#$} & {Residual} \\
    \hline
    \textit{MRCM} &13 & $8.8\times10^{-9}$ & 15 & $8.0\times10^{-9}$ & 16 & $2.1\times10^{-9}$\\
    \hline
    \textit{$2h$} & 3 & $7.5\times10^{-9}$  & 2 & $3.7\times10^{-9}$ & 3 & $8.1\times10^{-9}$\\
    \hline
    \textit{$4h$}  & 2 & $7.5\times10^{-9}$  & 2 & $8.3\times10^{-9}$ & 2 & $8.0\times10^{-9}$\\
    \hline
    \textit{$2h,2S$} & 2 & $7.8\times10^{-9}$ &  2 & $1.5\times10^{-11}$ & 2 & $8.4\times10^{-9}$\\
    \hline
    \textit{$2h,4S$} &  2 & $8.0\times10^{-9}$ &  1 & $6.1\times10^{-12}$ & 2 & $7.8\times10^{-9}$\\
    \hline
    \textit{$4h,2S$} &  2 & $7.8\times10^{-9}$ &  1 & $6.1\times10^{-12}$ & 2 & $8.2\times10^{-9}$\\
    \hline
    \textit{$4h,4S$} & 2 &$7.8\times10^{-9}$ &  1 & $9.4\times10^{-12}$ & 2 & $7.6\times10^{-9}$\\
    \hline
  \end{tabular}
  \caption{Number of iterations and corresponding final residuals for the problems with high-permeability inclusions, low permeability inclusions and both inclusions.}
  \label{table:num:fracture2}
\end{table}

From Tables~\ref{table:num:fracture1} and \ref{table:num:fracture2}, we observe that these problems require more iteration steps than those considered in the previous two sections. Figures~\ref{frac1} and~\ref {frac2} present the results for the three inclusion types, following the same format as in the earlier subsections, and clearly illustrate how the residual decreases over successive iterations.

\begin{figure}[H]
    \centering
    \includegraphics[width = 0.6\textwidth]{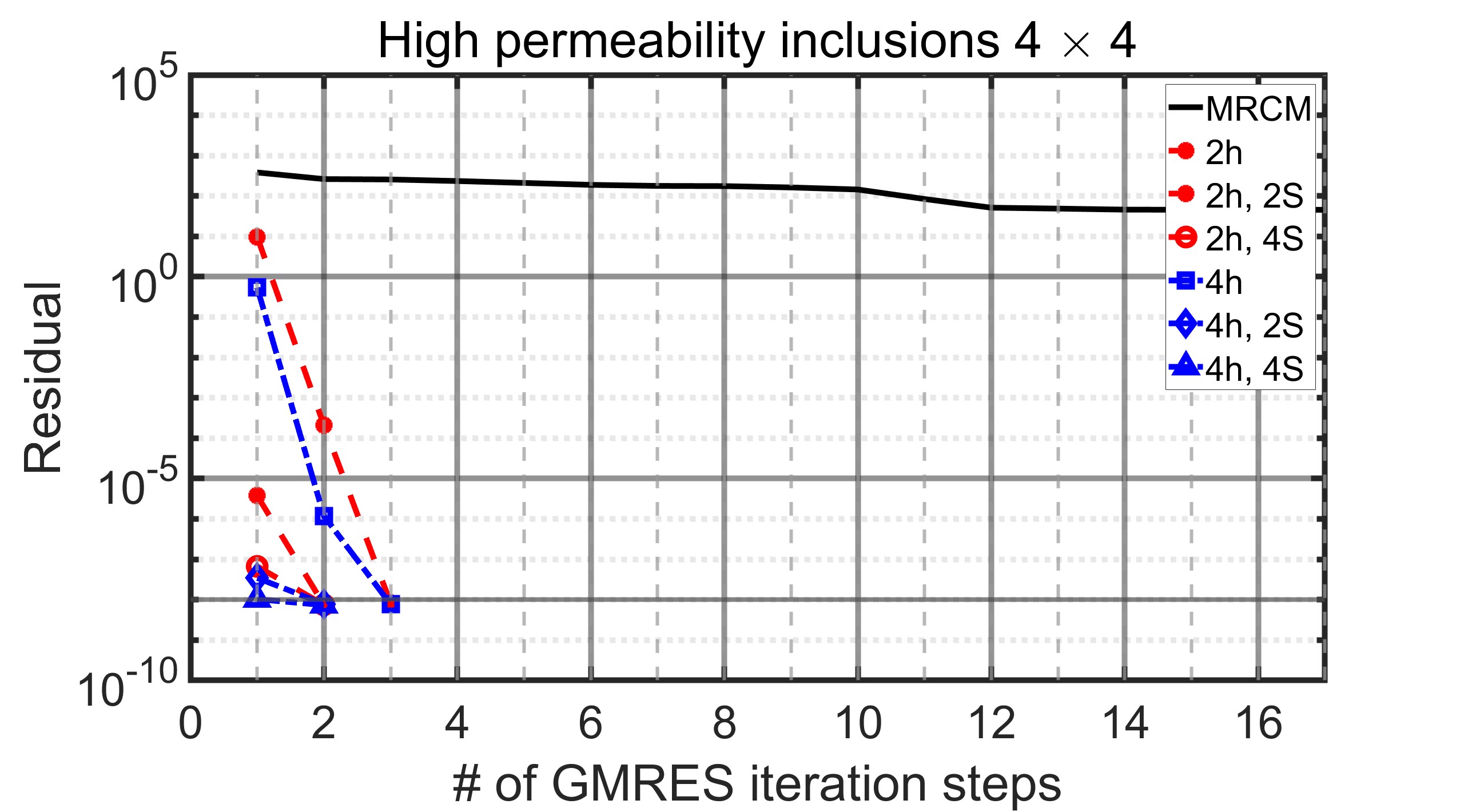}
    \includegraphics[width = 0.6\textwidth]{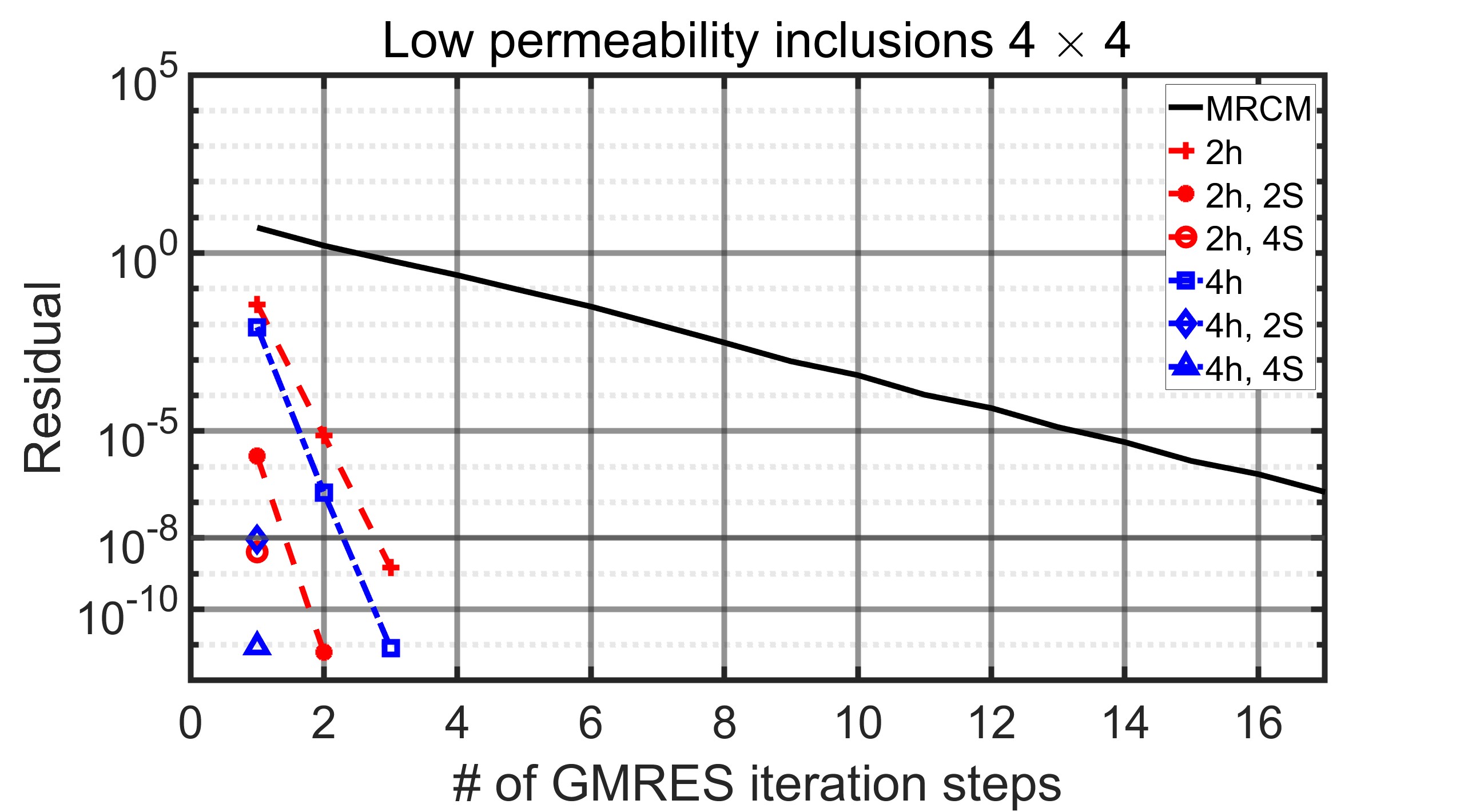}
    \includegraphics[width = 0.6\textwidth]{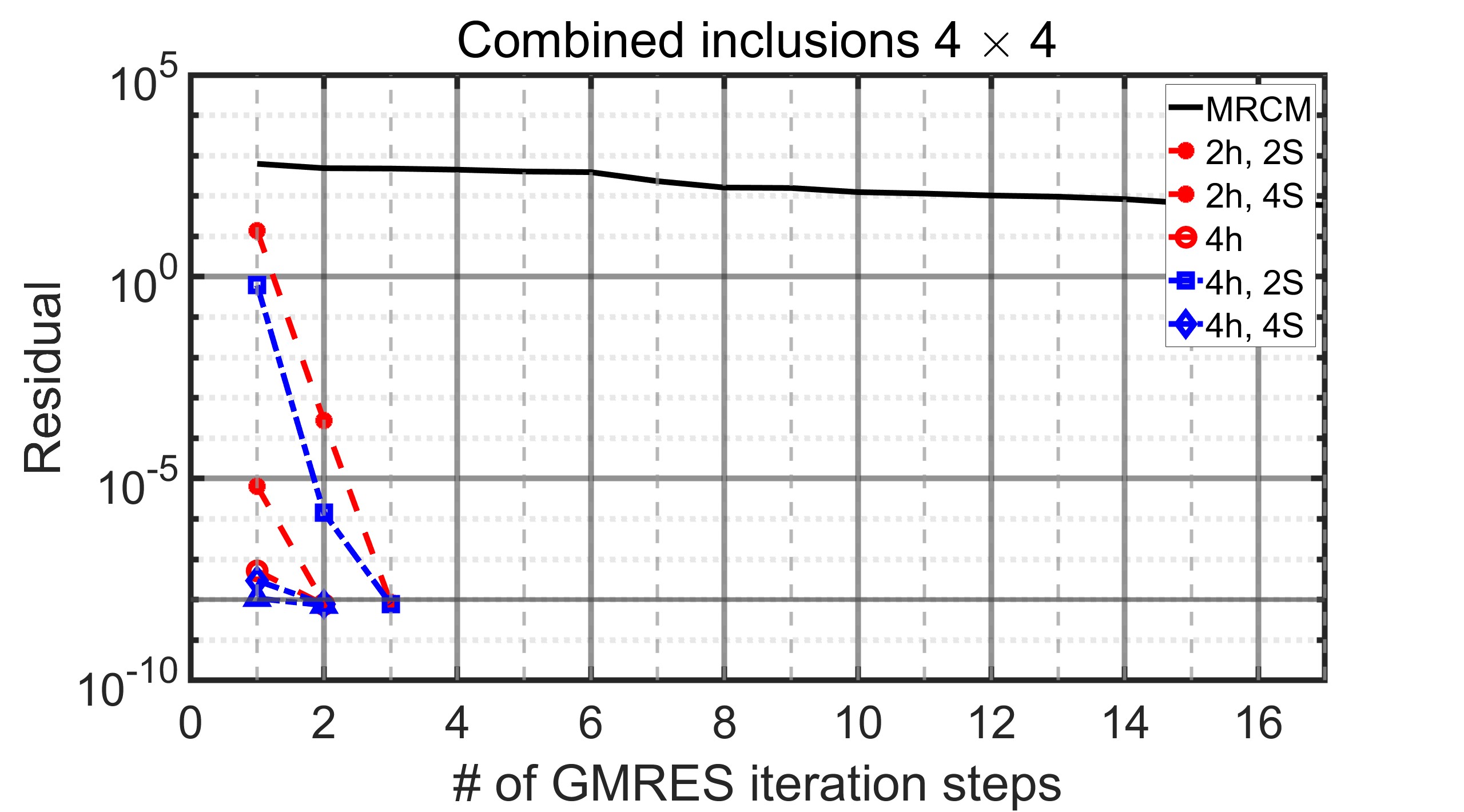}
    \caption{Number of iterations for the $4\times 4$ partition: high-permeability inclusions (top), low-permeability inclusions (middle), and combined inclusions (bottom).}
    \label{frac1}
\end{figure}

\begin{figure}[H]
    \centering
    \includegraphics[width = 0.6\textwidth]{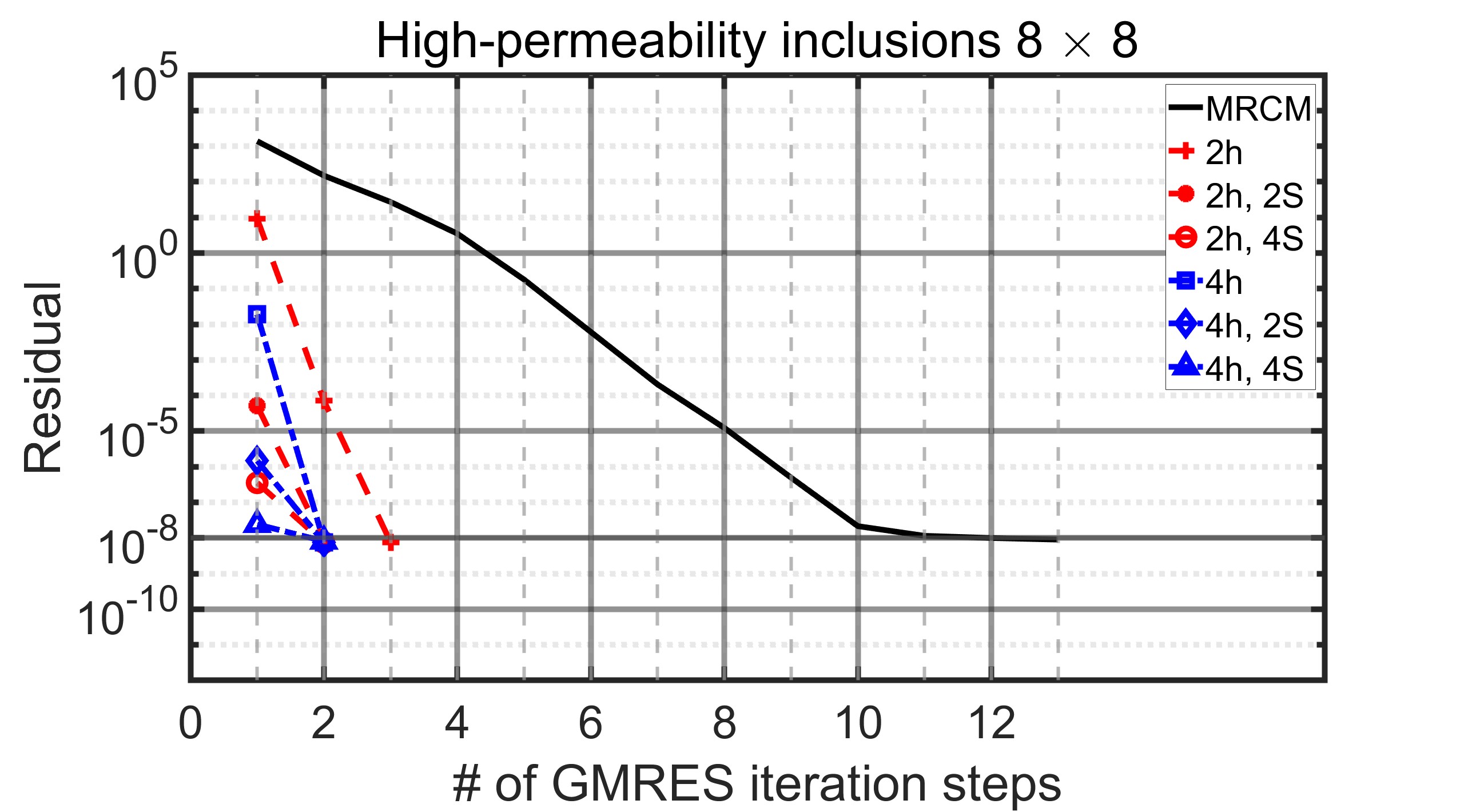}
    \includegraphics[width = 0.6\textwidth]{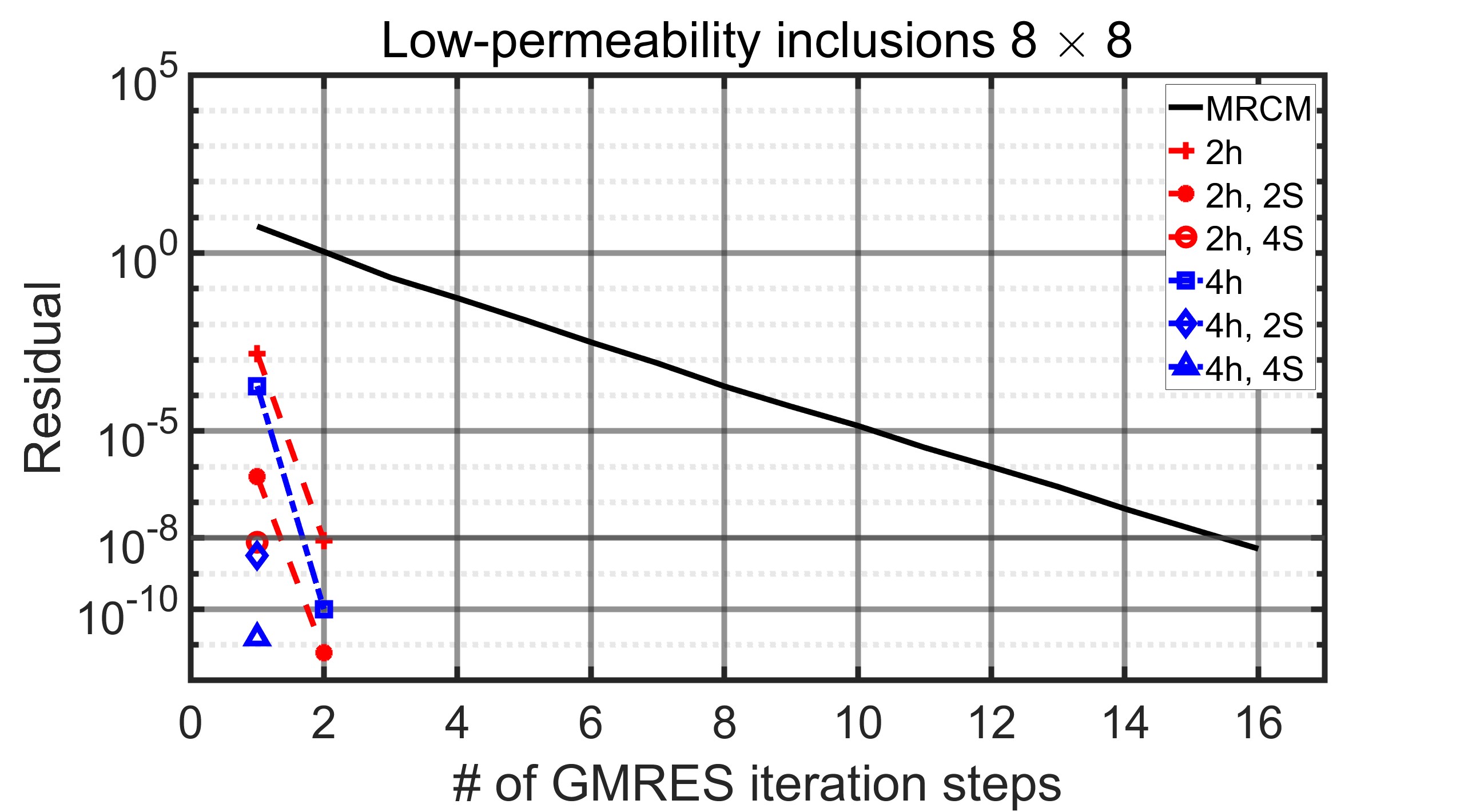}
    \includegraphics[width = 0.6\textwidth]{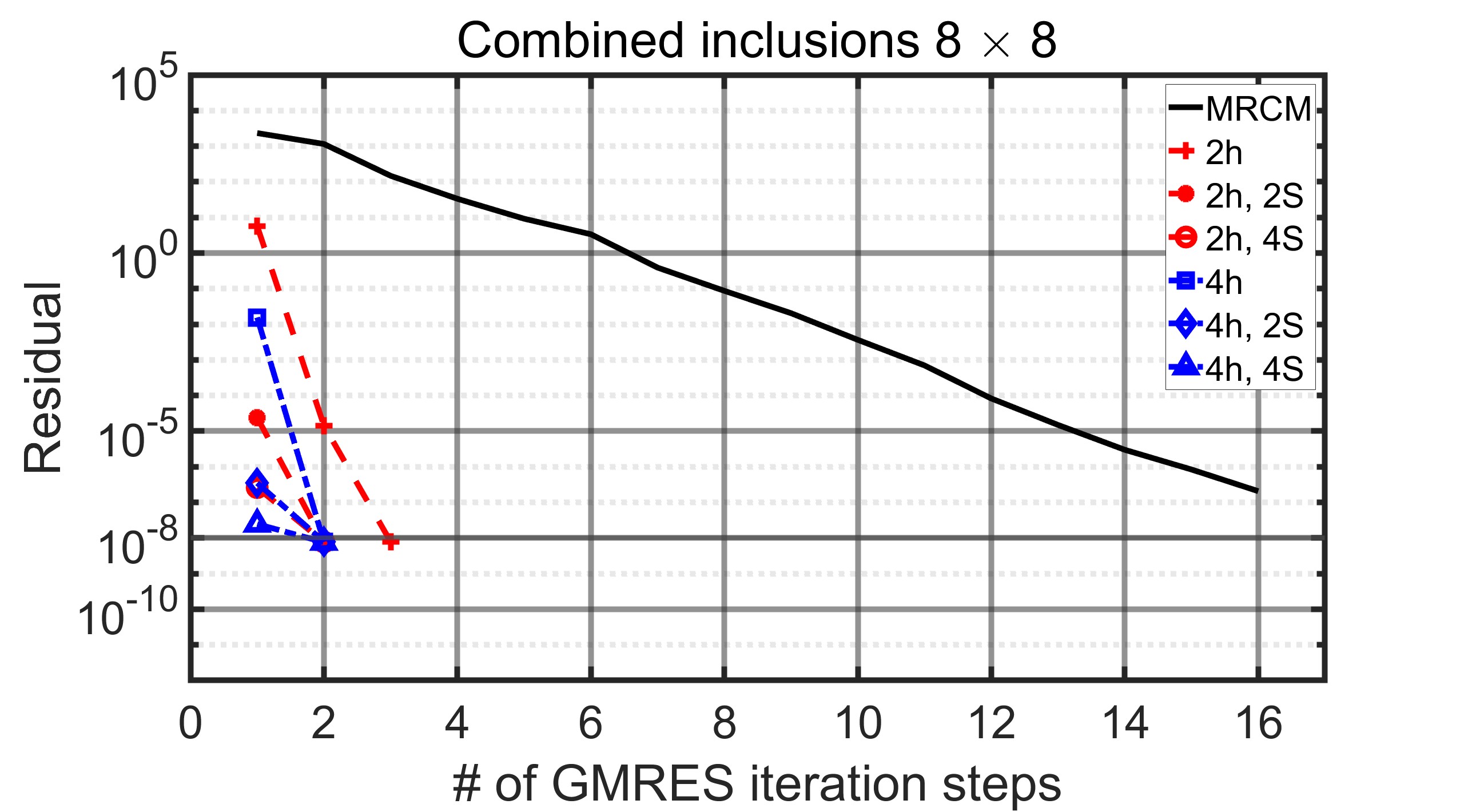}
    \caption{Number of iterations for the $8\times 8$ partition: high-permeability inclusions (top), low-permeability inclusions (middle), and combined inclusions (bottom).}
    \label{frac2}
\end{figure}

From the second picture in Figures ~\ref{frac1} and \ref{frac2}, we can clearly observe the significant improvement achieved by the smoothing procedure in the  low permeability inclusions case, indicating that smoothing has a more substantial effect on fields with such inclusions than on the other two cases. The results for the combined inclusion configuration are very close to those of the high permeability
inclusion case when both oversampling and smoothing are applied, suggesting that when the high- and low-permeability inclusions have the same contrast ratio relative to the background within the same field, the high-permeability inclusions dominate the final residual. Moreover, when the problem is reasonably well-conditioned, even for challenging permeability fields, our method can achieve the prescribed tolerance within $2$ iterations with only a $2h$ oversampling size and $2$ smoothing steps.
The comparison between Figs.~\ref{frac1} and \ref{frac2} further supports the conclusion from Section~\ref{subsec:SPE10} that increasing the number of subdomains reduces the number of iterations required to reach the prescribed tolerance.

We remark that, although the proposed method reduces the number of GMRES iterations, its computational cost depends on the number and size of local subdomain solves. Although MRCM-based procedures are fast and scalable, the present study focuses on iteration counts and residual reduction; a detailed wall-clock comparison, including setup and preconditioner application costs, is deferred to future work.

\section{Conclusions}

In this work, we propose a matrix-free right-preconditioning strategy for GMRES based on the MRCM-OS multiscale operator. The improved accuracy of MRCM-OS relative to the original MRCM translates into significantly faster convergence of the preconditioned system.

Our results demonstrate that the proposed approach substantially reduces iteration counts across a range of challenging, high-contrast subsurface flow problems. In Appendix~A, we outline a possible analytical framework for studying the method.

Future work will focus on establishing rigorous bounds for the preconditioner and assessing its computational efficiency in large-scale simulations. We also plan to extend the present formulation to a purely algebraic preconditioner, as well as to multiphase subsurface flow problems.

\appendix
\section*{Appendix A. A Convergence Framework for Right-Preconditioned GMRES}

We present a theoretical framework that explains the rapid convergence observed
in the numerical experiments. To analyze the convergence of GMRES applied to the
right-preconditioned system, we follow the classical Krylov subspace framework
of ~\cite{saad19861234}, together with field-of-values-based
convergence theory (see, e.g., \cite{FGMRES, greenbaum1997iterative, elman2014finite, simoncini2007recent}).

Let
\[
A p = b
\]
denote the fine-scale linear system, where $A \in \mathbb{R}^{n \times n}$ is
symmetric positive definite (SPD). Let $T : \mathbb{R}^n \to \mathbb{R}^n$
denote the linear operator corresponding to one application of the MRCM-OS
preconditioner.

Since GMRES is applied to the right-preconditioned system, we consider
\[
A T y = b, \qquad p = T y.
\]

We equip $\mathbb{R}^n$ with the energy inner product
\[
(x,z)_A := x^T A z, \qquad \|x\|_A := (x,x)_A^{1/2}.
\]

\subsection*{A.1 Classical spectral-equivalence context and approximate inverse behavior}

Classical domain decomposition theory establishes spectral equivalence
between the preconditioned operator and the identity in an appropriate energy
inner product. Specifically, one typically obtains bounds of the form
\[
c \, (x,x)_A \le (ATx, x)_A \le C \, (x,x)_A,
\qquad \forall x \in \mathbb{R}^n,
\]
for constants $0 < c \le C < \infty$.
Such estimates imply robustness of the preconditioned operator and lead to
condition number bounds.

However, these classical results do not by themselves explain the extremely
rapid convergence observed in Section~4. To capture this behavior, we introduce
a stronger assumption.

\medskip

\noindent
\textbf{Assumption A.1 (Approximate inverse property).}
There exists $\varepsilon \in (0,1)$ such that
\[
\|I - AT\|_A \le \varepsilon.
\]

By definition of the induced operator norm, this is equivalent to
\[
\|(I-AT)x\|_A \le \varepsilon \|x\|_A,
\qquad \forall x \in \mathbb{R}^n,
\]
i.e., the error operator $E := I - AT$ is a contraction in the energy norm.

This condition states that $T$ approximates $A^{-1}$ in operator norm and is
significantly stronger than classical spectral equivalence estimates.

\subsection*{A.2 Transformation to a symmetric setting}

Define the similarity transformation
\[
\widetilde B := A^{1/2} (AT) A^{-1/2}.
\]

\begin{lemma}
The matrices $\widetilde B$ and $AT$ are similar and therefore share the same
spectrum. Moreover,
\[
\|I - \widetilde B\|_2 = \|I - AT\|_A.
\]
\end{lemma}

\begin{proof}
For any matrix $M$,
\[
\|M\|_A
=
\sup_{x \neq 0} \frac{\|Mx\|_A}{\|x\|_A}
=
\|A^{1/2} M A^{-1/2}\|_2.
\]
Applying this identity to $M = I - AT$ yields the result.
\end{proof}

\subsection*{A.3 Field-of-values localization}

Let the field of values (numerical range) of $\widetilde B$ be defined by
\[
W(\widetilde B) = \{ x^* \widetilde B x : \|x\|_2 = 1 \}.
\]

\begin{lemma}
Under Assumption A.1, the field of values satisfies
\[
W(\widetilde B) \subset \{ z \in \mathbb{C} : |z - 1| \le \varepsilon \}.
\]
In particular,
\[
\operatorname{Re}(z) \ge 1 - \varepsilon > 0.
\]
\end{lemma}

\begin{proof}
For any $\|x\|_2 = 1$,
\[
|x^*(\widetilde B - I)x|
\le
\|\widetilde B - I\|_2
\le
\varepsilon.
\]
\end{proof}

\subsection*{A.4 GMRES convergence}

\begin{theorem}[GMRES convergence under the approximate inverse property]
Assume that $A$ is SPD and that Assumption~A.1 holds. Define
\[
\widetilde B := A^{1/2}(AT)A^{-1/2}, \qquad
\widetilde y := A^{1/2}y, \qquad
\widetilde b := A^{1/2}b.
\]

Then the GMRES residuals for the transformed system
\[
\widetilde B \widetilde y = \widetilde b
\]
satisfy
\[
\|\widetilde r_m\|_2 \le \varepsilon^m \|\widetilde r_0\|_2,
\]
where $\widetilde r_m := \widetilde b - \widetilde B \widetilde y_m$.

Consequently, to reduce the residual by a factor $\tau \in (0,1)$, it suffices
to take
\[
m \ge \frac{\log(\tau)}{\log(\varepsilon)}.
\]
\end{theorem}

\begin{proof}
By Lemma~1 and Assumption~A.1,
\[
\|I - \widetilde B\|_2 \le \varepsilon.
\]

GMRES satisfies the residual minimization property
\[
\|\widetilde r_m\|_2
\le
\min_{\substack{p \in \Pi_m \\ p(0)=1}}
\|p(\widetilde B)\|_2 \, \|\widetilde r_0\|_2,
\]
where $\Pi_m$ denotes the set of polynomials of degree at most $m$.

Choosing $p(z) = (1 - z)^m$, we obtain
\[
\|\widetilde r_m\|_2
\le
\|(I - \widetilde B)^m\|_2 \, \|\widetilde r_0\|_2
\le
\|I - \widetilde B\|_2^m \|\widetilde r_0\|_2
\le
\varepsilon^m \|\widetilde r_0\|_2.
\]
\end{proof}

\noindent
\textbf{Remark.} 
The estimate above is stated for the transformed system associated with
$\widetilde B = A^{1/2}(AT)A^{-1/2}$. It provides an abstract convergence
mechanism for the right-preconditioned operator. The numerical experiments
in Section~4 report the algebraic residual of the original system.

\subsection*{A.5 Interpretation for MRCM-OS}

The numerical results in Section~4 show that, for sufficiently enriched
MRCM-OS configurations, GMRES often converges in one or two iterations.
This behavior is consistent with the approximate inverse property in
Assumption~A.1.

For the MRCM-OS operator, the error $I - AT$ can be interpreted as the sum of:
\begin{itemize}
\item coarse-space approximation error,
\item oversampling localization error,
\item smoothing error.
\end{itemize}

Oversampling reduces interface truncation effects, while smoothing attenuates
high-frequency components. Together, these mechanisms suggest that $T$
behaves as a high-quality approximation to $A^{-1}$.

\medskip

\noindent
\textbf{Remark.}
Establishing a bound of the form $\|I - AT\|_A \le \varepsilon$ requires
approximation properties beyond classical condition number estimates and
remains an open problem for MRCM-OS.

\medskip

This framework explains the observed rapid convergence while clearly separating
what is rigorously justified from what is supported by numerical evidence.

\bigskip
\noindent {\bf Acknowledgments} 

This work is partially supported by the National Science Foundation (USA) grant 2401945, and by S\~ao Paulo research Foundation - FAPESP (Brazil) grants 2025/09743-3 and 2013/07375-0. Any opinions, findings, and conclusions or recommendations expressed in this material are those of the authors and do not necessarily reflect the views of the National Science Foundation or FAPESP. F.S.S. is partially funded by National Council for Scientific and Technological Development - CNPq (Brazil) grant 312372/2023-0.

The authors acknowledge the National Laboratory for Scientific Computing (LNCC/MCTI, Brazil) for providing HPC resources of the S. Dumont supercomputer, which have contributed to the research results reported within this paper (URL: http://sdumont.lncc.br). Additionally, the computing resources of the Cyber-Infrastructure Research Services at the University of Texas at Dallas Office of Information Technology were utilized. The authors would like to express their gratitude to Dr. M\'arcio Borges for his assistance in accessing the LNCC cluster. 

\bigskip
\noindent {\bf Declaration of generative AI and AI-assisted technologies in the writing process}

During the preparation of this work the authors used chatGPT in order to improve language and readability. 
After using this tool, the authors reviewed and edited the content as needed and take full responsibility for the content of the publication.

\bibliography{dilongref1,dilongref2,dilongref3,dilongref4,dilongref5}
\bibliographystyle{unsrt}

\end{document}